\documentclass{article}

\usepackage{amsmath,amssymb,amsfonts,theorem}
\usepackage{graphicx,color}
\usepackage{subfigure}
\graphicspath{{./Figures/}}

\newcommand{\until}[1]{\{1,\dots, #1\}}
\newcommand{\subscr}[2]{#1_{\textup{#2}}}

\newcommand{\setdef}[2]{\{#1 \; : \; #2\}}

\newcommand{\map}[3]{#1: #2 \rightarrow #3}

\newcommand{\bigintersect}{\bigcap}

\newcommand{\integers}{\mathbb{Z}} 
\newcommand{\reals}{\mathbb{R}} 

\newcommand{\real}{\ensuremath{\mathbb{R}}}

\newcommand{\realnonnegative}{\ensuremath{\reals}_{\geq0}}
\renewcommand{\natural}{{\mathbb{N}}}
\newcommand{\integer}{{\mathbb{Z}}}
\newcommand{\Z}{\mathbb{Z}}
\newcommand{\N}{\mathbb{N}}

\newcommand{\R}{\reals}

\newcommand{\eps}{\varepsilon}
\newcommand{\G}{G}    
\newcommand{\E}{E}    
\newcommand{\V}{V}    

\newcommand{\K}{\mathcal{K}}
\newcommand{\D}{\mathcal{D}}
\newcommand{\Eq}{\mathcal{E}}

\newcommand{\xave}{\subscr{x}{ave}}

\newcommand{\1}{\mathbf{1}} 
\newcommand{\diag}{\operatorname{diag}} 
\newcommand{\neigh}[1]{\mathcal{N}_{#1}}
\newcommand{\qd}{\texttt{q}}    
\newcommand{\qh}{\subscr{{\tt q}}{h}}     
\renewcommand{\S}{\mathcal{S}}    
\newcommand{\dist}{\rm dist}
\newcommand{\indeg}[1]{d_{\textup{in},#1}}
\newcommand{\outdeg}[1]{d_{\textup{out},#1}}
\newcommand{\dmax}{\subscr{d}{max}}
\newcommand{\interior}[1]{\operatorname{int}{(#1)}}

\newcommand{\dst}{\displaystyle}
\newcommand{\sym}{{\rm Sym}}

\DeclareMathOperator*{\bigcart}{\raisebox{-0.3ex}{\text{\Large$\times$}}}

\def\be{\begin{equation}}
\def\ee{\end{equation}}
\def\ba{\begin{array}}
\def\ea{\end{array}}
\def\eqa{\begin{eqnarray}}
\def\eqe{\end{eqnarray}}

{\theorembodyfont{\normalfont} 
\newtheorem{example}{Example}
\newtheorem{remark}{Remark}
}
\newtheorem{assumption}{Assumption}

\newtheorem{theorem}{Theorem}
\newtheorem{lemma}{Lemma}
\newtheorem{corollary}{Corollary}
\newtheorem{proposition}{Proposition}

\def\QED{\hfill \vrule height 7pt width 7pt depth 0pt\medskip}

\newcommand{\caratheodory}{Carath\'{e}odory\ }

\begin{document}
\title{Discontinuities and hysteresis in quantized average consensus
\thanks{An abridged version of this paper has been presented at the 8th IFAC Symposium on Nonlinear Control Systems, Bologna, Italy, September 2010.
The work of the first and third author is partially supported by
MIUR PRIN 2008 {\it Sistemi distribuiti su larga scala: stima,
ottimizzazione e controllo, con applicazioni.} The work of the
second author is partially supported by MIUR PRIN 2008 {\it
Advanced methods for feedback control of uncertain nonlinear
systems},  Progetto Ricercatori AST 2009 and a Johns Hopkins
University Applied Physics Laboratory grant.}
}

\author{Francesca Ceragioli\thanks{Dipartimento di Matematica, Politecnico di Torino, Torino, Italy.
\texttt{francesca.ceragioli@polito.it}.} \and 
Claudio De Persis\thanks{Lab. Mechanical Automation and Mechatronics, University of Twente, Enschede, The Netherlands and 
Dip. Informatica e Sistemistica, Sapienza Universit\`a di Roma, Roma, Italy.
\texttt{C.DePersis@ctw.utwente.nl}.} \and 
Paolo Frasca\thanks{Dipartimento di Matematica, Politecnico di Torino, Torino, Italy.
\texttt{paolo.frasca@polito.it}.}
}

\maketitle

\begin{abstract}
We consider continuous-time average consensus dynamics in which
the agents' states are communicated through uniform quantizers.
Solutions to the resulting system are defined in the Krasowskii
sense  and are proven to converge to conditions of ``practical
consensus''. To cope with undesired chattering phenomena we
introduce a hysteretic quantizer, and we study the convergence
properties of the resulting dynamics by a hybrid system approach.
\end{abstract}

\section{Introduction}
Communication constraints play a major role in consensus and
related problems of distributed computation and control. Such
constraints can be represented by a graph of available
communication links among agents, together with further
restrictions on which information can be exchanged across
links. Recently, the constraint of quantization, that is of
communication restricted to a discrete set of symbols, has
received significant attention. Although most  works  to-date
have dealt with discrete-time dynamics, it is worth considering
the same restrictions in the context of continuous-time dynamics, because the dynamics of
the agents is naturally described by continuous-time systems in many applications.
An example of this is robotic networks.
One might argue that it is possible to study the effect of quantization on
continuous-time systems by considering their discretized or sampled-data model.
However, implementing consensus control algorithms by discretizing the dynamics requires
implicitly that all the agents sample synchronously with the same
clock, a requirement which is difficult to  satisfy in practice. This
lack of synchronicity may disrupt the convergence properties of
the algorithm and thus asks for a different approach.
This paper proposes an approach which deals with
 consensus problems in continuous-time
without relying on sampled-data systems: consensus is achieved by quantized measurements which are transmitted asynchronously. Indeed, the quantization of the states induces a partition of the
 space into regions and  each agent
transmits quantized information about its state only when the
state crosses the boundary of the quantization regions. In this
sense communication among the agents takes place asynchronously.
Throughout the paper no assumption is imposed on the resolution of
the quantizers, which therefore can be also very coarse. Hence,
this paper specifically aims at giving a rigorous treatment of
continuous-time average consensus dynamics with uniform
quantization in communications. It is an established fact that
consensus problems can be reformulated  in terms of feedback
control systems: as expected, when quantization enters the loop,
the stabilization problem becomes more challenging. From a
mathematical point of view, a consequence of quantization is that
we obtain a system with a discontinuous righthand side, which is
not guaranteed to admit solutions in classical sense. This paper
proves  that classical or \caratheodory solutions  may in fact
not exist: considering solutions in some generalized sense is thus
unavoidable.

The literature provides
different approaches to the technical problem of having  systems
with discontinuous righthand sides (see~\cite{JC:08-csm,OH:79} for
a review of these topics). Tools from the theory of discontinuous differential equations   and nonsmooth analysis have already been applied in consensus problems in \cite{JC:06b,JC:08}.
Here we focus on Krasowskii solutions essentially for two reasons: effectiveness and generality.
With regard to effectiveness, there are many  results available concerning the  existence and
continuation of Krasowskii solutions, as well as a complete
Lyapunov theory \cite{JPA-AC:84,AB-FC:99}.
With respect to generality, since the set
of Krasowskii solutions includes Filippov and \caratheodory
solutions, results about Krasowskii solutions also hold for
Filippov and \caratheodory solutions, provided that they exist. On the
other hand, the set of Krasowskii solutions may be too large. In
particular, from a practical
point of view it may contain sliding modes which induce chattering phenomena.
In the context of quantized consensus, chattering amounts to fast
information transmission between the agents. This is undesirable
because it results in algorithms which require large bandwidth
communication channels to be implemented. To cope with this issue,
we propose the use of a quantizer endowed with a hysteretic
mechanism, and study the resulting dynamics by a hybrid system
approach.  Specifically, we provide an estimate of the data rate
needed to implement the quantized continuous-time consensus
algorithm. These results can be of interest to other application
fields, including load balancing problems and real-time control
systems.

With respect to earlier literature, our contribution  is twofold. On one hand, we give a mathematical treatment, in terms of differential equations with discontinuous righthand sides, of a continuous-time consensus system under uniformly quantized communication of the states. We do this when the communication graph  is only weakly connected and weight balanced.
After showing basic properties of solutions, such as existence, boundedness and average preservation, we prove convergence to a set containing the equilibria of the system. This set depends on the communication graph, and is reduced to the set of equilibria in the particular case of symmetric graphs.
On the other hand, our paper is  the first to propose the application of hysteretic quantizers to solve a consensus problem with data rate constraints.
Preliminarly, we prove that the system is well-posed, in the sense that a solution  exists, is
forward unique, and the set of switching times is locally finite. The main results consist in proving convergence and estimating the required data rate.

In our paper, we provide convergence results for both the Krasowskii and the
hysteretic dynamics discussed above. Due to the constraint of static uniform quantization we cannot obtain exact consensus, but we can obtain approximations of the consensus condition which we informally refer to as ``practical consensus''. Similar conditions have been obtained elsewhere in the literature, which has already considered some related problems of quantized consensus.
A number of publications have discussed --mostly in discrete-time systems-- several options to deal with the quantization constraint, namely uniform deterministic
quantizers~\cite{PF-RC-FF-SZ:08,AN-AO-AO-JNT:09}, uniform
randomized quantizers~\cite{TCA-MJC-MGR:08}, logarithmic
quantizers~\cite{RC-FB-SZ:10}, and adaptive
quantizers~\cite{TL-MF-LX-JFZ:09}. Moreover, quantization has been
considered also in {\it gossip} consensus
algorithms~\cite{AK-TB-RS:07,RC-FF-PF-SZ:10}. Regarding
continuous-time systems, relevant bibliography
includes~\cite{JC:06b}, which discusses
discontinuous differential equations with consensus applications, and~\cite{JY-SMLV-DL:08}, which studies a rendezvous algorithm in which each agent tracks  another agent assigned to it by a quantized control law.
A recent paper (\cite{DVD-KHJ:10}) is also related: using graph-theoretical tools, the authors study a quantized consensus problem, under the assumption that the communication graph is a tree.

The present  paper is organized as follows. In Section~\ref{sec:Problem} we
recall definitions and results from graph theory and about average
consensus dynamics without quantization, recalling how the
consensus problem can be reformulated in terms of a stability
problem. We also present the problem of quantization in a
continuous-time setting. In Section~\ref{sec:UnifQuant} we
introduce state quantization and Krasowskii solutions and we
study the fundamental properties of the system: we compute the set
of equilibria, we prove average preservation, and we deduce
asymptotic and finite-time convergence results. Then, in
Section~\ref{sec:HysteresisQuant} we define and study the system
under a hysteretic quantizer, in terms of a hybrid system. After
proving that chattering can not  occur, we study the
fundamental properties of this system: equilibria, average
preservation and  convergence. Some simulations are given
in Section~\ref{sec:Simulations} which illustrate our results and point to possible future research directions, which are discussed in our conclusions in Section 5.

{\bf Notations.} Given a subset $A$ of the Euclidean Space $\R^N$,
we denote as $\overline A$ its topological closure, by
$\interior{A}$ its interior and by $\partial A$ its boundary.
Given $N\in\N$, we let $\1$ ($\mathbf{0}$) be the $N\times 1$ vector whose
entries are 1 ($0$), $I$ be the $N$-dimensional identity matrix and
$\Omega=I-N^{-1}\1\1^*$, where the symbol $^*$ denotes conjugate transpose.
$\| \cdot \|$ denotes the Euclidean norm both for vectors and matrices.

\section{Preliminaries}\label{sec:Problem}

\subsection{Graph theory}
Let there be a weighted (directed) graph $\G=(\V,\E, A)$,
consisting of a node set $\V=\until{N}$, an edge set $\E\subset
\V\times\V$ and an adjacency matrix $A\in
\realnonnegative^{N\times N}$ such that $A_{ij}>0$ if $(j,i)\in
\E$, and $A_{ij}=0$ if
$(j,i)\not\in \E$. For every node, we define the set of its in-neighbors as $\neigh{i}=\setdef{j\in\V}{(j,i)\in \E}$. We assume no self-loops in the graph, that is 
$i\not\in\neigh{i}$ for every $i\in\V$.
Nodes (vertices) are referred to as agents, edges as links.
Let $\indeg{i}:=\sum_{j=1}^n A_{ij}$ and $\outdeg{j}=\sum_{i=1}^n
A_{ij}$ be, respectively, the in-degree and the out-degree of node
$i\in\V$. A graph is said to be {\em weight-balanced} if the
out-degree of each node equals its in-degree. Let $D=\diag(A\1)$
be the diagonal matrix whose diagonal entries are the in-degrees
of each node, which are equal to the number of incoming edges if the nonzero entries of the adjacency matrix $A$ are all
equal to $1$. Let $L=D-A$ be the Laplacian matrix of the graph
$\G$. Note that $L\1=\mathbf{0}$, and that $\1^*L=\mathbf{0}^*$ if
and only if $\G$ is weight-balanced. A path in a graph is an
ordered list of edges. Given an edge $(i,j)$, we shall refer to
$i$ and to $j$ as the tail and the head of the edge, respectively.
An oriented path is an ordered list of edges such that the head of
each edge is equal to the tail of the following one. The graph
$\G$ is said to be strongly connected if for any $i,j\in \V$ there
is an oriented path from $i$ to $j$ in $\G$. Instead, it is said
to be weakly connected if for each pair of nodes $i,j$ there
exists a path which connects $i$ and $j$. Observe that weakly
connected weight-balanced graphs are strongly connected
graphs~\cite[Proposition~2]{JC:08}. Recall the following result,
which can be derived from~\cite[Theorem~1.37]{FB-JC-SM:09} and
\cite[Formula~(1)~and~Section~2.2]{JC:08}.

\begin{lemma}\label{lemma:t1}
Let $\G$ be a weighted graph and suppose it is weight-balanced and
weakly connected. Let $L$ be its Laplacian matrix. Then:
\begin{description}
\item{(i)} The matrix $\sym(L):=\frac{L+L^\ast}{2}$ is positive semi-definite.
\item{(ii)} Denoted by $\lambda_2(\sym(L))$ the smallest
non-zero eigenvalue of $\sym(L)$,
$$
x^\ast \sym(L) x \ge \lambda_2(\sym(L))
\|x-\frac{\1\1^\ast}{N}x\|^2\;, $$ for all $x\in \R^N$, where
$\|\cdot\|$ denotes the Euclidean norm.
\end{description}
\end{lemma}

\bigskip
Sometimes, it may be convenient to restrict our attention to {\em
symmetric} graphs, that is graphs such that $A=A^*$. For a
symmetric graph there is no distinction between strong and weak
connectedness, so in that case we shall just say that the graph is
{\em connected}.

\subsection{Feedback consensus dynamics}

Let $\map{x}{\realnonnegative}{\real^N}$ be a time dependent
vector representing the agents' states. Its dynamics can be
written in terms of the control system
\begin{equation}\label{eq:control}
\dot{x}=u,
\end{equation}
where $x, u\in\R ^N$. Our aim is to construct a control law
$\map{u}{\R}{\R^N}$ such that, for all initial conditions,
solutions to~\eqref{eq:control} satisfy the {\em average
consensus} condition, that is
$$\lim_{t\to \infty}x(t)=\xave(0)\1,$$
where we let $\xave(t)=N^{-1}\1^*x(t).$ It turns out that one such
control  can be given in the feedback form $u=-Lx$ so that the
implemented system becomes
\begin{equation}\label{eq:StandAlgo}
\dot x(t)=-L x(t).
\end{equation}
Componentwise, this reads as
\begin{equation}\label{eq:StandAlgoCompwise}
\dot x_i(t)=\sum_{j\in \neigh{i}} A_{ij}\left(x_j(t)-x_i(t)\right),
\qquad \forall\,i\in \V
\end{equation}
where we recall $\neigh{i}$ is the set of the neighbors of agent $i$.

The following result, which can be deduced from~\cite{JC:08},
gives the weakest conditions for system~\eqref{eq:StandAlgo}
to converge to average consensus.
\begin{lemma}\label{lemma:StandCons}
If the weighted graph $\G$ is weakly connected and
weight-balanced, and $x(t)$ satisfies~\eqref{eq:StandAlgo}, then
$$\lim_{t\to
\infty}x(t)=\xave(0)\1.$$
\end{lemma}

In view of the above result, from now on we shall make use of the following standing assumption, unless otherwise stated.

\begin{assumption}\label{ass:a1}
The communication graph $\G$ is weakly connected and
weight-balanced.
\end{assumption}

\subsection{Quantized consensus dynamics}
It is clear that the dynamics~\eqref{eq:StandAlgo} is a rather idealized version of what can actually be implemented in a real control system. A very natural issue, as discussed in the introduction, is quantization.
Let $q$ be a quantizer, that is a map $\map{q}{\reals}{\S}$ with
$\S$ a discrete subset of $\reals$ (If $z\in \reals^d$, $q(z)$ is
meant componentwise). Several quantized dynamics based on~\eqref{eq:StandAlgo} are of interest from the point of view of the
applications: indeed, quantization may be inherent either to
communication or sensing among agents, or to the computation of
the feedback control, or to its application to the system. In
other words quantization can occur at the communication, sensor,
computation or actuator level. If the sensed/communicated data about the neighbors states is quantized, due to the use of a digital lossless channel or to finite-precision sensors, the resulting dynamics is
\begin{equation}\label{eq:StateQuant}
\dot x(t)=-L q(x(t)).
\end{equation}

Despite its interest, the dynamics (\ref{eq:StateQuant}) has not received much attention yet, although its discrete-time counterpart has been widely studied, for instance in~\cite{AN-AO-AO-JNT:09,PF-RC-FF-SZ:08}.
In the present paper, we shall focus on this dynamics, considering two different quantizers. First, we take a uniform static quantizer; then, we design a hysteretic quantizer which prevents chattering.

\section{Krasowskii quantized dynamics}\label{sec:UnifQuant}
Let us consider a uniform quantizer
$\map{\qd}{\reals}{\Delta\integers}$, defined by
$$\qd (z)=\left\lfloor \frac{z}{\Delta}+\frac{1}{2} \right\rfloor\Delta.$$
Note that $|\qd(z)-z|\le \frac{\Delta}{2}$. Moreover if $x\in
\reals^N$, we let $\qd (x)= (\qd (x_1),\ldots,\qd (x_N))^*$.
We shall then consider the quantized dynamics
\begin{equation}\label{eq:qAlgo}
\dot x=-L\qd(x).
\end{equation}
Since the righthand side of~\eqref{eq:qAlgo} is discontinuous, it
is important to specify in which sense solutions have to be
intended.
One definition of solution is in the sense of {\em
\caratheodory}: an absolutely continuous function $x(t)$ is a
\caratheodory solution to~\eqref{eq:qAlgo} on an interval $I$ if
it satisfies (\ref{eq:qAlgo}) for almost every $t\in I$. Such
solution is said to be complete if  $I= [0, +\infty )$. However,
this natural definition is not suitable to study the system at
hand, because of the following fact:

\begin{proposition}[\caratheodory solutions]\label{prop:CarSol}
There are weakly connected and weight-balanced
graphs and   initial conditions for which no \caratheodory
solution to~\eqref{eq:qAlgo} exists. Moreover, in
these cases, there is a positive-measure set of initial
conditions such that \caratheodory solutions starting from that
set  are not complete.
\end{proposition}

\begin{proof}
The righthand side of~\eqref{eq:qAlgo} is constant over open hypercubes with edge of length $\Delta$, and surfaces of discontinuity are hyperplanes.
Each hyperplane is orthogonal to one of the elements of the
canonical basis of $\reals^N$, $e_1,\dots,e_N$. Let us consider
one such hyperplane, namely one which is orthogonal to $e_j$ and
let us denote it by $S_j$. Let us consider a point $\hat x\in
S_j$, i.e. such that $\hat x_j=(k+\frac 12)\Delta$ for some $k\in
\integers$, and such that if $i\neq j$,
then $\hat x_i\neq (h+\frac 12)\Delta$ for any $h\in \integers$.
Let $I(\hat x)$ be a neighborhood of $\hat x$ such that, for any
$x\in I(\hat x)$, it holds that $x_j\ne (h+\frac{1}{2})\Delta$ for
any integer $h\ne k$, and   $x_i\not = (h+\frac
12)\Delta$ if $i\neq j$ for any $h\in \integers$. Let $I^+(\hat x)=\{ x\in
I(\hat x): x_j> (k+\frac 12)\Delta\} $ and $I^-(\hat x)=\{ x\in
I(\hat x): x_j< (k+\frac 12)\Delta\} $. Let $q$ be the value of
$\qd(x)$ in $I^+(\hat x)$, $f^+$ be the value of $-L \qd(x)$ in
$I^+(\hat x)$ and $f^-$ be the value of $-L \qd(x)$ in $I^-(\hat
x)$.
Then, letting the nonzero entries of the adjacency
matrix to be all equal to $1$, the components $f^+$ and $f^-$
which are orthogonal to $S_j$ have simple expressions and
interpretation:
\begin{align*}
f^+_\perp &:=\langle f^+,e_j\rangle=f^+_j=\sum_{k\in\neigh{j}}{q_k}-\indeg{j} q_j,\\
f^-_\perp &:=\langle f^-,e_j\rangle=f^-_j
=f^+_j+\indeg{j}\Delta,
\end{align*}
where $\indeg{j}$ is the in-degree of node $j$. If either $f^+_\perp$ is
positive or $f^+_\perp<-\indeg{j}\Delta$, then $f^+_\perp$ and $f^-_\perp$ have the same sign, otherwise  they have opposite sign.
A numerical example for the latter case is as follows: let $\indeg{j}=2$, $q_j=2\Delta$ and
$\setdef{q_k}{k\in \neigh{j}}=\{\Delta,2\Delta\}$, then $f^+_j=-\Delta$ and
$f^-_j=\Delta$. In this case, $f^+_\perp<0$ and $f^-_\perp>0$, so that there are no \caratheodory solutions starting at $\hat x$.
Moreover, solutions starting in a (sufficiently small) neighborhood of $\hat x$ reach the surface in finite time, and can not be extended further.
\end{proof}

\begin{example}\label{ex1}
Consider a graph with adjacency matrix whose entries are:
$A_{12}=A_{21}=1,A_{23}=A_{32}=1,
A_{11}=A_{22}=A_{33}=A_{13}=A_{31}=0$ and the initial condition
$\hat x=(\Delta, \frac{3}{2}\Delta, 2\Delta)$.
Clearly the graph satisfies Assumption~\ref{ass:a1}.
Note that the righthand side of \eqref{eq:qAlgo} is discontinuous
at $\hat x$, being on the plane $x_2=\frac{3}{2}\Delta$. Following
the proof of Proposition \ref{prop:CarSol} one can easily check
that $f^+_2=-\Delta$ and $f^-_2=\Delta$. A possible
\caratheodory solution issuing from $\hat x$ could not leave the
surface $x_2=\frac{3}{2}\Delta$ in the direction of decreasing
$x_2$ because $f^-_2=\Delta>0$ nor could leave the surface in the
direction of increasing $x_2$, since $f^+_2=-\Delta<0$. The only
possibility for a solution would be to remain on the surface but such a solution would  not satisfy the
definition of a \caratheodory solution. Hence, there are no
\caratheodory solutions issuing from $\hat x$.
\end{example}

In view of Proposition~\ref{prop:CarSol}, in the sequel of this paper we shall consider solutions in a more general sense, which is due to Krasowskii.
An absolutely continuous function $x(t)$ is a Krasowskii solution to \eqref{eq:qAlgo} on an interval $I$  if it satisfies at almost every $t\in I$ the differential inclusion
\begin{equation}\label{eq:qAlgo-inclusion}
\dot x\in \K(-L\qd(x)),
\end{equation}
where
$$\K(-L\qd(x))=\bigintersect_{\delta >0}{\overline {\rm co}}\, (-L \qd (B(x,\delta))),$$
and $B(x,\delta)$ is the Euclidean ball of radius $\delta$ centered in $x$.
The solution is said to be complete if $I=[0, +\infty )$.
Note that, thanks to Theorem~1 in \cite{BP-SSS:87}, we have
\begin{equation}\label{Kcartesianproduct}
\K(-L\qd(x))=-L\K(\qd(x))\subseteq -L(\bigcart_{i\in V}\K(\qd(x_{i}))=\bigcart_{i\in V}(\sum_{j\in \neigh{i}} A_{ij}\left(\K(\qd (x_j))-\K(\qd (x_i))\right))
\end{equation}
 where $\bigcart_i$ denotes  the Cartesian product of the sets indexed by $i$.
Representations of the
map $\qd$ and of the set-valued map $\K\qd$ are given in
Figure~\ref{fig:Kq-map}.

\begin{figure*}[htbp]
\begin{center} \begin{tabular}{cc}
    \includegraphics[width=.49\textwidth]{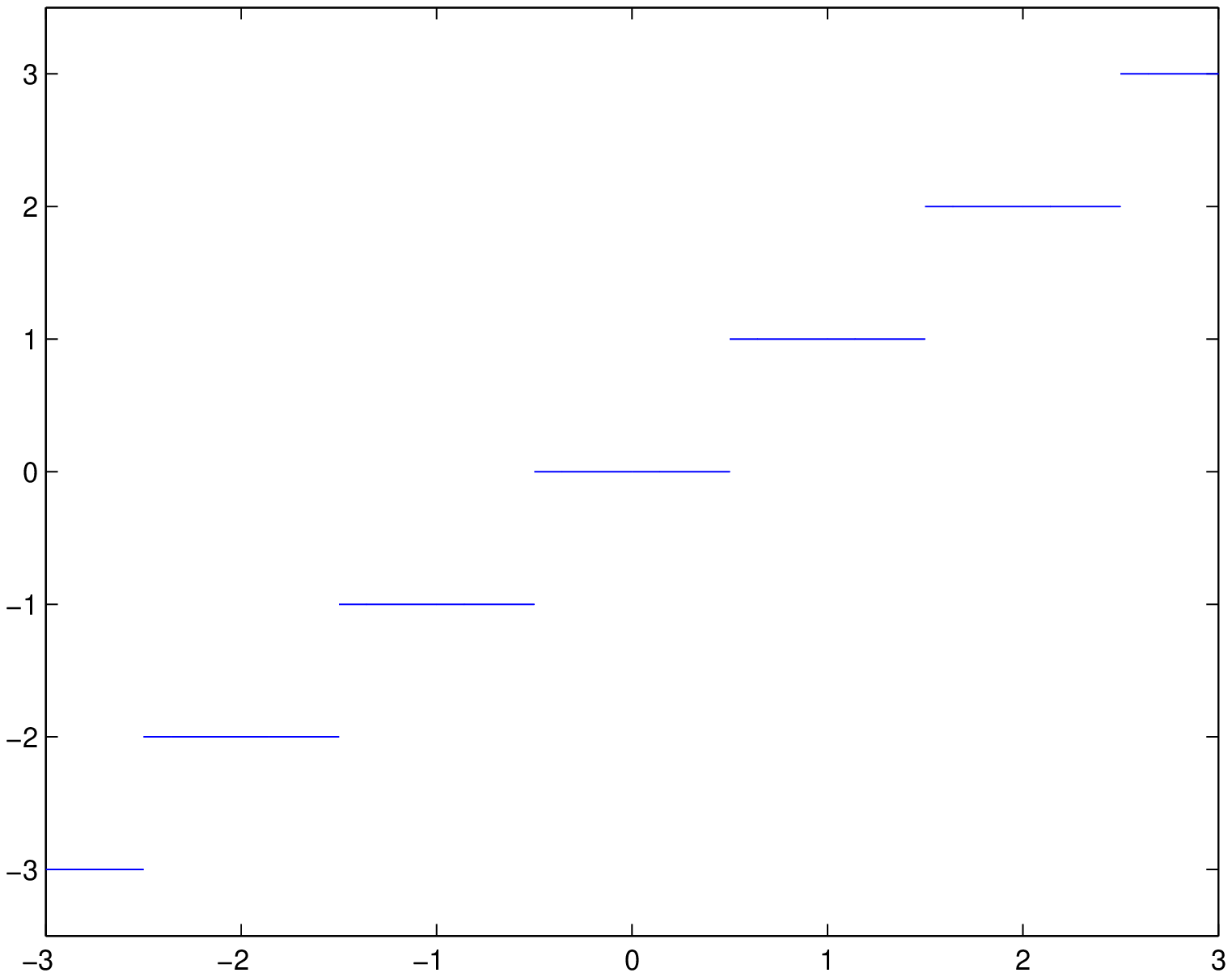}&
    \includegraphics[width=.49\textwidth]{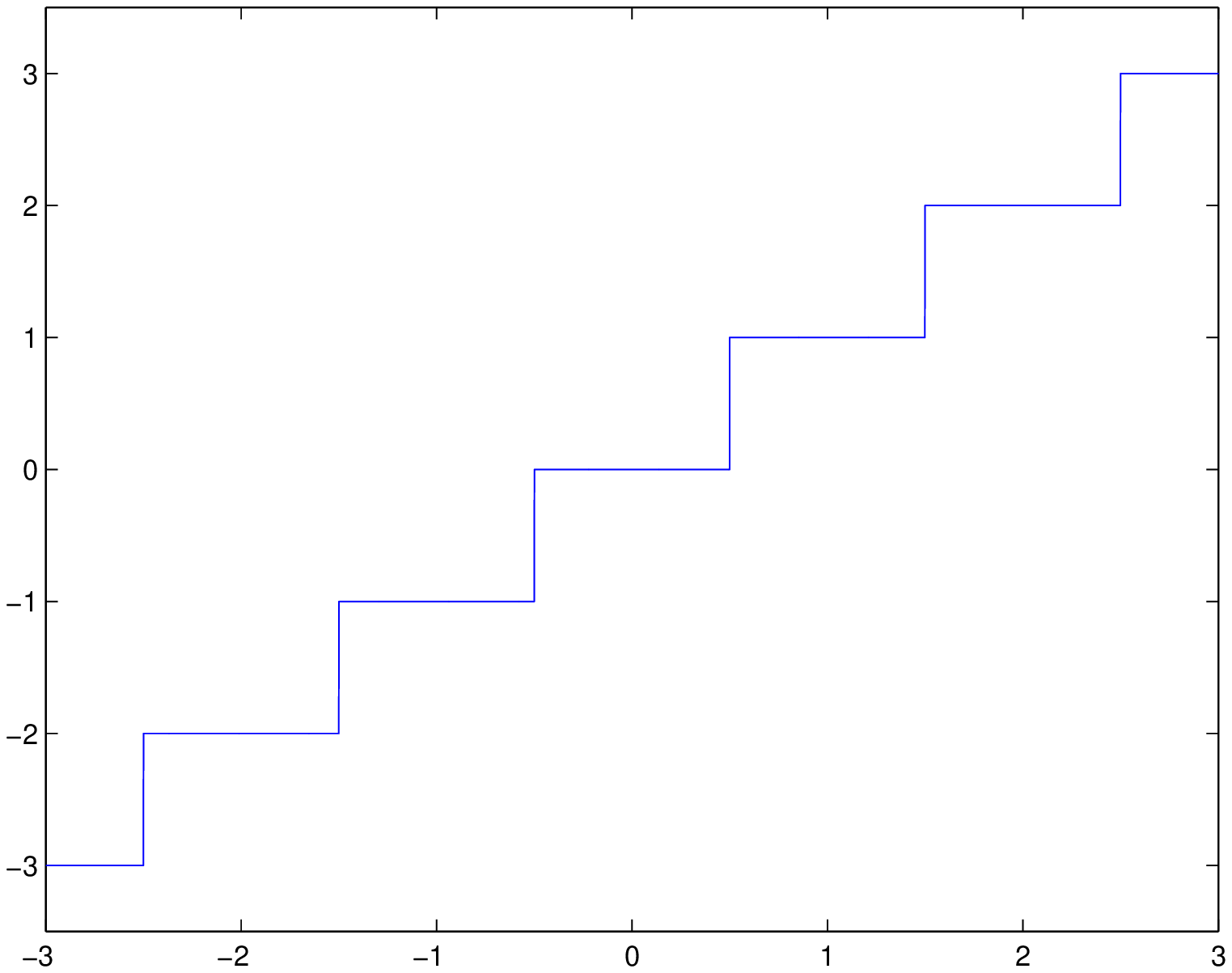}
    \end{tabular}
    \caption{Visualization of the map $\qd(x)$ and the set-valued map $\K \qd(x)$, when $\Delta=1$.}\label{fig:Kq-map}
\end{center} \end{figure*}

\begin{lemma}[Krasowskii solutions] \label{lemma:Krsol} For any $x^0\in \R^N$ there exists a complete Krasowskii solution $x(t)$  to \eqref{eq:qAlgo} such that $x(0)=x^0$.
\end{lemma}

\begin{proof}
Local existence of Krasowskii solutions is guaranteed by the fact that the righthand side of \eqref{eq:qAlgo} is measurable and locally bounded (see \cite{OH:79}). By standard arguments completeness of solutions can be deduced by their boundedness (see, e.g., \cite{VNVS:60}).
We then prove that {\em solutions are bounded}.
Let $x(t) $ be a Krasowskii solution to \eqref{eq:qAlgo} such that $x(0)=x^0$.
Let $m(t)=\min\{ x_i(t), i\in V  \}$ and $M(t)=\max\{ x_i(t), i\in V \}$.  Let $t_0\geq 0$ be fixed.
We prove that $m(t)\geq m(t_0)$ for all $t\geq t_0$. Analogously it can be proven that $M(t)\leq M(t_0)$ for all
$t\ge t_0$. By taking $t_0=0$ we will then deduce that for any $t> 0$ and for any $i\in V$ one has $m(0)\leq x_i(t)\leq M(0)$, i.e. $x(t)$ is bounded.
We first consider the case $m(t_0)\not =(k-\frac{1}{2})\Delta$ for
any $k\in \mathbb Z$. Assume by contradiction that there exists
$\bar t >t_0$ such that $m(\bar t)<m(t_0)$ and let $\bar i\in V$
be such that $x_{\bar i}(\bar t)=m(\bar t)$. Note that $m(t)$ is a
continuous function since it is the minimum of a finite number of
continuous functions.  Then
there exists $\delta> 0$ such that $m(\bar
t-\delta)=m(t_0)$ and for all $t\in (\bar t-\delta, \bar t)$ one
has $\qd (m(t))=\qd (x_{\bar i}(t))=\qd (m(t_0))$ and $\qd
(x_i(t))\geq\qd (m(t_0))$ for all $i\in V$. We remark that for all
$t\in (\bar t-\delta, \bar t)$ it holds $\K (\qd(x_{\bar
i}(t)))=\{ \qd (m(t_0)) \}$ and $v\geq \qd (m(t_0))$ for all $v\in
\K (\qd(x_j(t)))$ and any $j\in{V}$. Since $\dot x_{\bar i}(t)\in
(-L\K(\qd (x(t)))_{\bar i}\subseteq \sum_{j\in \neigh{\bar i}}
A_{ij}\left(\K  \qd (x_j(t))-\K \qd (x_{\bar i}(t))\right)$, we
get that $\dot x_{\bar i}(t)\geq 0$ for almost all $t\in (\bar
t-\delta, \bar t)$. On the other hand, $x_{\bar i}(\bar t)=m(\bar
t) < m(t_0)=m(\bar t-\delta)\leq x_{\bar i}(\bar t-\delta)$. Then
there must exist a subset $I$ of $(\bar t-\delta, \bar t)$ such
that $I$ has positive measure and $\dot x_{\bar i}(t)< 0$ for
almost all $t\in I$, i.e.\ a contradiction.
We now consider the case in which $m(t_0) =(k-\frac{1}{2})\Delta$ for some  $k\in \mathbb Z$.
Let $t^*=\inf \{ t>t_0: m(t)<m(t_0)\}$, and assume by contradiction that $t^*<+\infty$. Note that $m(t)\geq m(t_0)$ for $t\in [t_0,t^*)$ and, since $m$ is continuous, one also has $m(t^*)=m(t_0)$.
Let us fix any $\bar t>t^*$ such that
$(k-\frac 32)\Delta<m(t)< m(t_0)$ for any $t\in (t^*,\bar t)$.
We take a sequence
$\{t_n\}_{n\in \mathbb N}$ such that $t_n\in (t^*,\bar t)$ for any $n$ and $t_n\to t^*$.
Since $m(t_n)<m(t_0)$,
thanks to  the  previously analyzed case, we get that $m(t)\geq m(t_n)$ for all $t\geq t_n$. By passing to the limit we get that $m(t)\geq m(t^*)=m(t_0)$ for all $t\geq t^*$ and finally we have that $m(t)\geq m(t_0)$ for all $t\geq t_0$,
 again a contradiction.
\end{proof}

We observe that Lemma~\ref{lemma:Krsol} does not guarantee uniqueness of solutions, as explained in the following remark.
\begin{remark}[Sliding mode]\label{rem:sliding-mode}
From the proof of Proposition~\ref{prop:CarSol}, we see that
if \caratheodory solutions originate from any point $\hat x\in
S_j$ such that $f^+_\perp<0$ and $f^-_\perp>0$, then they do not exist, while
Krasowkii solutions do. Krasowskii solutions starting in a
neighborhood of $\hat x$ reach the surface in finite time and
slide on it. We remark that such solutions are not backward
unique. A general discussion about these behaviors  can be found
in~\cite[page~51]{AFF:88}.
\end{remark}
%

\addtocounter{example}{-1}

\begin{example}[Cont'd]
Consider again the system in Example \ref{ex1}. It is
straightforward to check that
\[
\K(-L\qd(\hat x))=\overline{{\rm co}}
\left\{
\left(\ba{r}\Delta\\ -\Delta\\ 0\ea\right),
\left(\ba{r}0\\ \Delta\\ -\Delta\ea\right)
\right\}
=
\left\{
v\in \real^3\,:\, v=\left(\ba{c}\lambda\Delta\\ (1-2\lambda)\Delta\\ -(1-\lambda)\Delta\ea\right),\; \lambda\in [0,1]
\right\}.
\]
Hence a Krasowskii solution issuing from $\hat x$ is the solution
to $\dot x_1 = \frac{\Delta}{2}$, $\dot x_2 = 0$, $\dot x_3 =
-\frac{\Delta}{2}$ with initial condition $\hat x$. On the other
hand, as explained previously in Example \ref{ex1}, no \caratheodory solution issuing from $\hat x$ exists.
\end{example}


\medskip
A preliminary result establishes that the average of the states is
preserved by Krasowskii solutions to~\eqref{eq:qAlgo}: this fact will be a key step to
obtain many of the following results.
We recall that this result holds under the standing Assumption~\ref{ass:a1}: the same is true for all the following ones, with the exception of Theorem~\ref{theor:convergenceD}
and Proposition~\ref{prop:FiniteTimeConv} which require a stronger assumption.

\begin{lemma}[Average preservation]\label{lem:AveragePreserved}
Let $x(t)$  be a Krasowskii solution to \eqref{eq:qAlgo}. Then, $\xave(t) = \xave(0)$ for
all $t\ge 0$.
\end{lemma}
\begin{proof}
By definition, $\xave(t)= N^{-1}\1^*x(t)$, and then
$\dot x _{\rm ave}(t)=N^{-1}\1^*\dot x(t).$
By \eqref{eq:qAlgo-inclusion},
$$\dot x _{\rm ave}(t)\in N^{-1}\1^*\K(-L\qd(x(t)))\  {\rm for \,  a.e.t}. $$
Thanks to  the first equality in (\ref{Kcartesianproduct}) and the fact
that $\1^*L={\bf 0}^*$,
$$\dot x _{\rm ave}(t)\in  -N^{-1}\1^*L \K(\qd(x(t)))=\{ 0\}, $$ then  we get that $\dot x _{\rm ave}(t)=0 $ for a.e. $t$  and
finally that $\xave (t)$ is constant.
\end{proof}

\subsection{Graph-dependent convergence results}
A first set of results regards the limit behavior of
system~\eqref{eq:qAlgo}, which depends on
the quantizer and the graph topology. The following proposition
proves convergence of solutions to a certain set.
\begin{theorem}[Convergence]\label{theor:ConvergenceStrip}
If $x(t)$ is any Krasowskii solution to~\eqref{eq:qAlgo} and
$$M=\setdef{x\in\reals^N} {\frac{1}{\sqrt{N}}\|x-\xave(0)\1
\|\le\frac{||L||}{{\lambda_{2}(\sym(L))}}\frac{\Delta}{2}},
$$
 then
${\dist}(x(t),M)\to 0$ as $t\to +\infty$.
\end{theorem}
\begin{proof}
Let $y(t)=\Omega x(t)=x(t)-\xave(t)\1 $. Then, $\dot y=\Omega \dot x\in \Omega \K(-L \qd(x)).$ Since $\K(\Omega(-L\qd(x)))=\Omega \K(-L\qd(x)) =\K(-L\qd(x))=-L\K \qd(x),$
we have $\dot y\in - L \K\qd(x).$
Consider the function $V(y)=\frac{1}{2}y^*y$ and let  $v\in\K\qd(x)$. Note that if $v\in \K\qd(x)$, then $\| v-x\|\leq
\sqrt N\frac{\Delta}{2}$. We have that
\begin{align*}
\nabla V(y)\cdot \dot y
=&-y^*Lv\\
=&-y^* L(x+v-x)\\
=& -y^*Lx-y^*L(v-x)\\
=& -y^*Ly-y^*L(v-x)\\
=& -y^*\sym(L)y-y^*L(v-x)\\
\le & -\lambda_2(\sym(L)) ||y||^2 + ||y||\,||L||\dst\frac{\Delta}{2} \sqrt{N}\\
=& -\lambda_2(\sym(L))||y|| \left(||y|| -
\dst\frac{||L||}{\lambda_2(\sym(L))} \dst\frac{\Delta}{2}
\sqrt{N}\right),
\end{align*}
where the inequality in the second-last line follows from
Lemma~\ref{lemma:t1}. This implies  convergence to the set
$$\setdef{y\in\reals^N}{\|y\|\le \frac{||L||}{\lambda_{2}(\sym(L))}\frac{\Delta}{2} \sqrt{N}}.$$
Finally the statement follows from average preservation in
Lemma~\ref{lem:AveragePreserved}.
\end{proof}

\smallskip

With a slight extension of the argument leading to Theorem~\ref{theor:ConvergenceStrip}, we can prove finite-time convergence to a set larger than the set $M$ in Theorem~\ref{theor:ConvergenceStrip}, and provide an estimate of the convergence speed.

\begin{corollary}[Finite-time convergence]\label{c9}
If $x(t)$ is any Krasowskii solution to~\eqref{eq:qAlgo}, then for
any $\eps\in(0,1)$ there exists a finite time $T(\varepsilon)$
such that $x(t)$ belongs to the set
\[
M(\varepsilon)=\setdef{x\in\reals^N}{\frac{1}{\sqrt{N}}\|x-
\xave(0)\1 \|\le \dst\frac{1}{1-\varepsilon}
\frac{||L||}{\lambda_2(\sym(L))}\frac{\Delta}{2}}.
\]
for all $t\ge T(\varepsilon)$.
\end{corollary}

\begin{proof}
Consider, as in the proof of Theorem~\ref{theor:ConvergenceStrip},
the differential inclusion $\dot y \in -L Kq(x)$, the function
$V(y)=\frac{1}{2}y^*y$ and let $v\in\K\qd(x)$. Then, as
before,
\begin{align*}
\nabla V(y)\cdot \dot y\le& -\lambda_2(\sym(L))||y||
\left(||y|| - \dst\frac{||L||}{\lambda_2(\sym(L))}
\dst\frac{\Delta}{2} \sqrt{N}\right)\;.
\end{align*}

From the latter, we see that, if
\begin{equation}\label{eq:Ineq-y}
||y(t)||  > \dst\frac{1}{1-\varepsilon}
\dst\frac{||L||}{\lambda_2(\sym(L))} \dst\frac{\Delta}{2}
\sqrt{N}\;,
\end{equation}
then
\begin{align*}
\nabla V(y)\cdot \dot y < &
-\varepsilon\lambda_2(\sym(L))||y||^2\\
= &-2\varepsilon\lambda_2(\sym(L))V(y)
\end{align*}

If at some time $t_0$ the condition~\eqref{eq:Ineq-y} is satisfied
by $y(t_0)$, then there exists $t>t_0$ such that the Lyapunov
function computed along the trajectories of~\eqref{eq:qAlgo}
satisfies
\[
V(y(t))\le {\rm e}^{-2\varepsilon\lambda_2(\sym(L)))(t-t_0)}
V(y(t_0))
\]
and therefore
\[
||y(t)|| \le {\rm e}^{-\varepsilon\lambda_2(\sym(L)))(t-t_0)}
||y(t_0)||\;.
\]
Assuming without loss of generality that $\|y(t_0)\|\not = 0$, from the latter inequality we conclude that, if at time $t_0$ the
condition~\eqref{eq:Ineq-y} is satisfied, then there exists a time
\begin{equation}\label{eq:Teps-unif}
T(\varepsilon)=\max \left\{ 0, \dst\frac{-1}{\varepsilon \lambda_2(\sym(L))}
 \ln \left( \dst\frac{1}{1-\varepsilon}
\dst\frac{||L||}{\lambda_2(\sym(L))} \dst\frac{\Delta }{2}
\dst\frac{\sqrt{N}}{||y(t_0)||} \right) \right\}
\end{equation}
such that $||y(t)||$ satisfies
\begin{equation*}
||y(t)|| \le \dst\frac{1}{1-\varepsilon}
\dst\frac{||L||}{\lambda_2(\sym(L))} \dst\frac{\Delta}{2}
\sqrt{N}\;,
\end{equation*}
for any $t\ge t_0+T(\varepsilon)$. The thesis then follows,
recalling the definition of $y(t)$ and
Lemma~\ref{lem:AveragePreserved}.
\end{proof}

The above convergence results assert that the error induced by
quantization, with respect to the non-quantized consensus
dynamics~\eqref{eq:StandAlgo}, can be made arbitrarily small by
decreasing the quantization error. Up to a $\sqrt{N}$ factor, due
to the length of the vector, the committed error is proportional
to $\frac{||L||}{\lambda_{2}},$ and then depends on the network.
It is plain that it would be of interest to state a result of
convergence to a stronger practical consensus condition, in which
the committed error does not depend on the network topology but
only on the quantizer precision. This issue is the topic of the
next paragraph.


\subsection{Equilibria}
In this paragraph, we shall describe the equilibria of the
system~\eqref{eq:qAlgo}, which depend on the quantizer precision
only. Hence, proving convergence to equilibria turns out to be a
way to prove a practical consensus condition which does not depend
on the network.
The following proposition characterizes the equilibria of the system. We recall that $x_0$ is a {\em (Krasowskii) equilibrium} if the function $x(t)\equiv x_0$ is a (Krasowskii) solution, that is if $\mathbf{0}\in \K(-L\qd(x_0))$.
Let
$$
{\D }=\{ x\in \reals^N: \exists k\in \integer\ \  such \ that\   \qd (x_i)=\Delta k,\,  \forall \,  i\in \V\}.
$$
\begin{proposition}[Equilibria]\label{prop:equilibria}
The set of Krasowskii equilibria of~\eqref{eq:qAlgo} is
$\overline\D$.
\end{proposition}
\begin{proof}
Note that $L\qd(x)$ is zero in $\D$, and is discontinuous on the
boundary of $\D$. Let us define the set of Krasowskii equilibria
as $E=\{ x\in \reals ^N: \mathbf{0}\in -L\K(\qd(x))\}$, and let
$\tilde E=\{ x\in \reals ^N: \mathbf{0}\in -L(\times
_i\K(\qd(x_{i})))\}$, where  we recall that $\times_i$ denotes the
Cartesian product of the sets indexed by $i$. Since Theorem~1
in~\cite{BP-SSS:87} implies that $ \K(\qd(x))\subseteq \times
_i\K(\qd(x_{i}))$, we have that $E\subseteq \tilde E$.
In order to prove that $\D = E$, we will first prove that
$\interior{\D} \subset E$. Since $E$ is closed due to the
fact that the set-valued map $ \K(\qd(\cdot ))$ is upper semicontinuous (see, e.g., the definition of upper semi-continuity
given in \cite{KD:92}), then also $\overline \D \subseteq
E\subseteq \tilde E$. Later we will prove that $\tilde E\subseteq
\overline \D$. These two facts imply that $\overline \D = E= \tilde E$, and namely our statement.\\
Let us prove that $\D \subset E$. Let us assume $x_0\in
{\rm int\,} \D$. Since $\qd$ is continuous at $x_0$, then
$\K(\qd(x_0))=\{\qd(x_0)\}$ and  $\K(-L\qd(x_0))=\{-L\qd(x_0)\}=\{-L(\Delta k\1)\}=\{-\Delta
kL\1\}=\{{\bf 0}\}$, i.e. $x_0\in E$. The points $x\in
\overline\D$  also belong to $E$  thanks to the fact that
$E$  is closed.
Then, let us prove that $\tilde E\subseteq \overline \D$. $x_0\in
\tilde E$ if there exists $v\in \times _i  \K(\qd(x_{0_i}))$
such that $Lv=0$. This  is equivalent to the fact that $\ker L\cap
\times _i  \K(\qd(x_{0_i}))\not =\emptyset$. Since $\ker
L={\rm span} \1$, there exists $v\in \ker L\cap \times _i  \K(\qd(x_{0_i}))$ if there exists $\lambda\in \R$ such that
$v=\lambda \1$, i.e.  $v_i=\lambda $ for any $i\in \V$ and $v_i\in
\K(\qd(x_{0_i}))$. Such $\lambda$ can be either $\lambda
=\Delta k$ for some $k\in \integers$, or $\lambda \not =\Delta k$
for any $k\in \integers$. In the first case, for any $i\in\V$ we have
that
$\{ x_i: \lambda\in \K(\qd (x_i)) \}=\left[
(k-\frac{1}{2})\Delta, (k+\frac{1}{2})\Delta\right]$. In the
second case, we have that
$\{ x_i: \lambda\in \K(\qd (x_i)) \}=\{(\lfloor
\frac{\lambda}{\Delta}\rfloor +\frac{1}{2} )\Delta  \}$. Finally
we get that if $v=\lambda \1\in \times _i  \K(\qd(x_{0_i}))$
then for every $i\in\V$ either $x_{0_i}\in \left[
(k-\frac{1}{2})\Delta, (k+\frac{1}{2})\Delta\right]$ for some
$k\in \integers$, or $x_{0_i}={(k+\frac 12)\Delta}$ for some $k\in
\integers$, i.e. $x_0\in \overline\D$.
\end{proof}

\bigskip

We now prove that $\overline\D$ is strongly invariant, i.e.
there are no  trajectories exiting $\overline\D$.
\begin{proposition}[Strong invariance]\label{lemma:StrongInvariance}
If $x(t)$ is a Krasowskii solution to~\eqref{eq:qAlgo} such that
$x(0)\in \overline\D$, then $x(t)\in \overline\D$ for
all $t\geq 0$.
\end{proposition}

\begin{proof}
Recall that $$\overline \D=\setdef{x\in \R^N}{\exists k\in \Z \text{ s. t. }
(k-\frac12)\Delta\le x_i\le (k+\frac12)\Delta, \: \forall i\in\V
}.$$ Let $x(t)$ be a solution to~\eqref{eq:qAlgo} such that
$x(0)=x^0\in\overline\D.$ If $x^0\in \interior\D$, then
$L\qd(x^0)=\bf 0$, and thus $\K L \qd(x^0)=\{\bf 0\}.$ Hence
$\interior\D$ is invariant.
Let then $x^0\in \partial\D$. Then there exist
$k^0\in \integers$ and $\V ^-, \V^+\subseteq \V$ (not both empty) such that $x_i(0)= (k-1/2)\Delta$ for all $i\in \V ^-$, $x_i(0)= (k^0+1/2)\Delta$ for all $i\in \V ^+$, and $x_i(0)\in ((k^0-1/2)\Delta, (k^0+1/2)\Delta)$ for all $i \in
\V\setminus ( \V ^-\cup \V^+)$.
Let us assume by contradiction that there exists $T>0$ such that
$x(T)\not \in \overline\D$, i.e. either there exists
$\overline i\in \V$
such that  $x_{\overline i}(T)>(k^0+1/2)\Delta$
or there exists
$\overline j\in \V$ such that $x_{\overline j}(T)<(k^0-1/2)\Delta$.
For brevity, we examine only the former case.
Let $i^*$ be such that $x_{i^*}(T)=\max \{ x_i(T), i\in \V\}$.
Without loss of generality, we can
assume that $(k^0+3/2)\Delta >x_{i^*}(T)> (k^0+1/2)\Delta$
and $x_i(T)< (k^0+3/2)\Delta$ for all $i\in \V$.
Thanks to the continuity of $x(t)$, there exists $T'<T$ such that
$x_{i^* }(T')= (k^0+1/2)\Delta$ and
$(k^0+1/2)\Delta<x_{ i ^*}(t)<(k^0+3/2)\Delta$ for all $t\in
(T',T)$.
 Since $x_{i^*}(T)-x_{i^*}(T')>0$ there exists a subset $\mathcal{T}$ of
$(T',T)$ such that $\mathcal{T}$ has positive Lebesgue measure and
for all $t\in \mathcal{T}$ the derivative $\dot x_{i ^*}(t)$
exists and is positive.
On the other hand for all $t\in (T',T)$
one has $(k^0+1/2)\Delta<x_{i^*}(t)<(k^0+3/2)\Delta$, which
implies $\qd (x_{ i^*}(t))=(k^0+1)\Delta$ for all  $t\in (T',T)$. Let
us now consider any $v\in \K(\qd (x(t))$ with $t\in (T',T)$.
It holds $v_{ i^*}=(k^0+1)\Delta$ and $v_j< (k^0+1)\Delta$ for all $j\not ={ i^*}$, and therefore $(-Lv)_{ i^*}< \indeg{i^*}(k^0+1)\Delta-\indeg{i^*}(k^0+1)\Delta=0$. From this fact it follows that $\dot x_{ i ^*}(t)< 0$, i.e.\ a contradiction.
\end{proof}

\bigskip
Next, we provide a second convergence result, stating that on any connected {\em symmetric} graph the quantized dynamics converges to the set of equilibria. This fact implies that the error induced by quantization does not depend on the network properties, but only on the quantizer.
%
\begin{theorem}[Convergence to equilibria]\label{theor:convergenceD}
If the weighted graph $\G$ is symmetric and $x(t)$ is any
Krasowskii solution to~\eqref{eq:qAlgo}, then ${\dist}(x(t),\D)\to 0$ as $t\to\infty$.
\end{theorem}
\begin{proof}
We consider  the function $V(x)=\frac{1}{2}x^*x$ and we prove that
$${\nabla V(x)\cdot (-Lv)\leq 0}$$ for every $x$ and for every
$v\in   \K\qd(x)$. Let us first remark that $\K\qd(x)\subseteq \times_i  \K(\qd(x_{i})$ and then $v\in
\K\qd(x)$ implies that $v\in \times _i  \K(\qd(x_{i}))$.
\begin{align*}
\nabla V(x)\cdot (-Lv)=&-x^*Lv\\
=&\sum _i x_i\sum _j A_{ij}(v_j-v_i)\\
=&-\frac{1}{2}\sum _{ij} (x_j-x_i)A_{ij}(v_j-v_i)\leq 0,
\end{align*}
where in the third equality we have used the symmetry of the
graph, i.e. of the matrix $A$, and where  the last inequality is a
consequence of the fact that if $a,b\in \reals, a\leq b, $ then
for all $u\in  \K\qd(a)$ and all $w\in  \K\qd(b)$ it
holds $u\leq w$. Namely, all terms of the above summation are
nonnegative.
Let $Z=\{ x\in \reals ^N:\   \exists v\in   \K\qd(x) \text{
such that }   \nabla V(x)\cdot (-Lv)=0 \} $ and let $M$ be the
largest weakly invariant subset of $\overline{Z}$. By LaSalle
invariance principle for differential inclusions (see
\cite{AB-FC:99}, Theorem~3) we have that  any solution $x(t)$ to
$\dot x\in -L(\K(\qd(x))$ is such that
 ${\dist}(x(t),M)\to 0$ as $t\to +\infty$.

We now prove that $Z\subseteq \overline\D$. More precisely
we prove that if $\bar x \not \in \overline\D$ then $\bar x
\not \in Z$. Proving  $\bar x \not \in Z$ is equivalent to prove
that  for all $v\in \K\qd(\bar x)$ one has
$\nabla V(\bar x)\cdot (-Lv)<0$, i.e. $\sum _{ij}
(\bar x_j-\bar x_i)A_{ij}(v_j-v_i)>0$,
and, for this to hold, it is sufficient to prove that there exist
$i,j$ such that $A_{ij}\not =0$ and $(\bar x_j-\bar x_i)(v_j-v_i)\not =0$.
Let us then assume that $\bar x\notin \overline\D$.

This means that for all $k\in\integers$ there exists ${\overline i}\in \V$ such that
$\overline x_{\overline i}<(k-1/2)\Delta $ or  there exists ${\overline j}\in \V$ such that
$\overline x_{\overline j}>(k+1/2)\Delta $. Let
$k_1=\qd (\overline x_1)$. Then there exists $i\in\V$ (without loss generality we assume $i=2$) such that $\overline x_2<(k_1-1/2)\Delta$ or $\overline x_2>(k_1+1/2)\Delta$. We consider the case
$\overline x_2<(k_1-1/2)\Delta$. Let $k_2$ be such that $\qd (\overline x_2)=k_2\Delta$. Clearly $k_2<k_1$.
We examine the following possible cases:
\begin{itemize}
\item[(a)] $\bar x_1\not =(k_1-1/2)\Delta$ and $\bar x_2\not =(k_2-1/2)\Delta;$
\item[(b)] $\bar x_1 =(k_1-1/2)\Delta$ and $\bar x_2\not =(k_2-1/2)\Delta;$
\item[(c)] $\bar x_1\not =(k_1-1/2)\Delta$ and $\bar x_2 =(k_2-1/2)\Delta;$
\item[(d)] $\bar x_1 =(k_1-1/2)\Delta$ and $\bar x_2 =(k_2-1/2)\Delta.$
\end{itemize}
In case (a) for all $v\in \K\qd(\bar x)$ one has
$v_1=k_1\Delta$ and $v_2=k_2\Delta$ then $v_1-v_2=\Delta
(k_1-k_2)\not =0$ since $k_1\not =k_2$. In case (b) we remark
that, since  $\bar x \not \in \overline\D$, then $k_2\leq
k_1-2$. For all $v\in \K\qd(\bar x)$, one has $v_1=\alpha k_1\Delta +(1-\alpha )(k_1-1)\Delta= \Delta
(k_1-1+\alpha )$ with $\alpha \in [0,1]$, and $v_2=k_2\Delta$.
Then $v_1-v_2= \Delta(k_1-k_2-1+\alpha )\geq \Delta\not =0$.
Analogously in case (c) we get $v_1=\Delta k_1$, $v_2=\alpha (k_2-1)\Delta+(1-\alpha)k_2\Delta$, then $v_1-v_2=\Delta (k_1-k_2+\alpha)$
with $\alpha \in [0,1]$, then $v_1-v_2\not =0$.
Finally in case (d) we also have $k_2\leq k_1-2$, and moreover,
since $\K \qd (\bar x)\subseteq \times _i\K(\qd (x_{i}))$,
for all $v\in \K \qd (\bar x)$ we have
\begin{align*}
v_1=&\Delta [\alpha _1(k_1-1)+(1-\alpha _1)k_1] \\
v_2=&\Delta [\alpha _2(k_2-1)+(1-\alpha _2)k_2],
\end{align*}
with $\alpha _i\in [0,1]$, $i=1,2$. Then
we get $v_1-v_2=\Delta[k_1-k_2-\alpha _1+\alpha_2]\not =0$.
Hence we have proved that $Z\subseteq \overline\D$. This
fact also implies that $\overline Z\subseteq \overline\D$,
and finally,
${\dist}(x(t),\overline\D)\to 0$ as $t\to+\infty$.
\end{proof}

\smallskip

We remark that Theorem~\ref{theor:convergenceD} does not imply that solutions converge to a point in $\D$. However, Lemma~\ref{lem:AveragePreserved} implies that solutions whose initial conditions belong to the hyperplane
$\sum_{i=1}^{N}x_i=N (k+\frac 12)\Delta$ for some $ k\in \integers$  converge to the
point $(k+\frac 12)\Delta\1 $.

One may ask whether solutions to~\eqref{eq:qAlgo} reach $\overline\D$ in finite time. This claim is supported by
numerical simulations, which we will illustrate in
Section~\ref{sec:Simulations}. Moreover, the following result
shows that the claim is true for almost any initial condition.

\begin{proposition}[Finite-time convergence to equilibria]\label{prop:FiniteTimeConv}
If the weighted graph $\G$ is symmetric and $x(t)$ is any
Krasowskii solution to~\eqref{eq:qAlgo} such that $\xave(0)\not
=(k+1/2)\Delta$ for every $k\in \integers$, then there exists $T$ such
that $x(T)\in \overline\D$.
\end{proposition}
\begin{proof} Let us consider any Krasowskii solution $x(t)$ to~\eqref{eq:qAlgo} such that
$x(0)\not \in \overline\D$ and its composition with the
function $V(x)= \frac 12 x^*x$. For almost all $t\geq 0$ we have
that $\dot x(t)=-L v(t)$ with $v(t)\in  \K\qd(x(t))$ and
(see the proof of Theorem~\ref{theor:convergenceD})
as long as $x(t)\not\in\overline\D$, we also have that for every $t\ge0$,
\begin{align*}
 \frac{d}{dt}V(x(t)) =&\nabla V(x(t))\cdot \dot x(t)\\
=& -\frac{1}{2}\sum_{ij} (x_j(t)-x_i(t))A_{ij}(v_j(t)-v_i(t))\\
\le& -\frac12 \underline{A} \sum_{(i,j)\in\E} (x_j(t)-x_i(t))(v_j(t)-v_i(t))\\
=& -\frac12 \underline{A} \sum_{(i,j)\in\E} |x_j(t)-x_i(t)||v_j(t)-v_i(t)|\\
\le& -\frac12 \underline{A} \max_{(i,j)\in\E} |x_j(t)-x_i(t)||v_j(t)-v_i(t)|\\
\le& -\frac12 \underline{A} \Delta \max_{(i,j)\in\E} |x_j(t)-x_i(t)|\\
\le & - \phi\Delta \underline{A},
\end{align*}
where $\displaystyle\underline{A}=\min_{i,j\in\V}\setdef{A_{ij}}{A_{ij}\not= 0}$ and
$\displaystyle\phi=\frac12\inf_{t\geq 0}\{\max_{(i,j)\in \E}{|x_j(t)-x_i(t)|}:\   x(t) \in \reals ^N\setminus \overline\D\}.$
If $\phi >0$, by integrating over the interval $[0,T]$, we get
\[ V(x(T))-V(x(0))\leq -\phi \underline{A}\Delta T.\]
Assume by contradiction that $x(t)$ does not reach $\overline\D$ in finite time. Then by letting $T\to +\infty$  we get that
$\lim _{T\to +\infty}V(x(T))=-\infty$, which contradicts the fact
that $V$ is positive definite.

We now prove 
that if $\phi=0$, then $\xave(0)=(k+1/2)\Delta$ for some $k\in \integers$. If
$\phi=0$, then for any $n\in\natural$ there exists $t_n$ such that $x(t_n)
\in \reals ^N\setminus \overline\D$ and, for all $(i,j)\in
\E$, it holds $|x_j(t_n)-x_i(t_n)|<\frac 1n$.
Let us fix any $k\in\V$, and consider a path of length $K$ connecting all the vertices of $\G$ and starting from $k$. Let us denote the states of the nodes in this path as $\tilde
x_1,...,\tilde x_K$, noting that some of the nodes may
appear in the list more than once. From the fact that $\phi =0$ we
deduce that
$$\tilde x_1(t_n)-\frac 1n<\tilde x_2(t_n)<\tilde x_1(t_n)+\frac 1n,$$
$$\tilde x_1(t_n)-\frac 2n<\tilde x_2(t_n)-\frac 1n<\tilde x_3(t_n)<\tilde x_2(t_n)+\frac 1n<\tilde  x_1(t_n)+\frac 2n,$$
$${\rm etc.}$$
$$\tilde x_1(t_n)-\frac Kn<\tilde x_K(t_n)<\tilde x_1(t_n)+\frac Kn.$$
By summing $N$ of these inequalities corresponding to the $N$
different nodes we get, since by definition $\tilde x_1=x_k$,
that
$$x_k(t_n)-\frac 1n \sum_{i=1}^{K}i<\frac 1N \sum_{i=1}^{N}x_i(t_n)<x_k(t_n)+\frac 1n \sum_{i=1}^{K}i.$$
Since $\frac 1N \sum_{i=1}^{N}x_i(t_n)=\frac 1N
\sum_{i=1}^{N}x_i(0)$ we obtain
$$\left|\frac 1N \sum_{i=1}^{N}x_i(0)-x_k(t_n)\right|<\frac{K(K+1)}{2 n}$$
and then $x_k(t_n)\to \frac 1N \sum_{i=1}^{N}x_i(0)$ as
$n\to+\infty$. By the arbitrariness of $k$, we get $x(t_n)\to \xave (0)\1 $
as $n\to+\infty$.
We recall that  $x(t_n) \in \reals ^N\setminus \overline\D$
for every $n$, then $\xave (0)\1 \in \overline{\reals ^N\setminus
\overline\D}$. Finally since $\overline{\reals ^N\setminus
\overline\D}\,\cap\, {\rm span} \1=\setdef{(k+\frac
12)\Delta}{k\in \integers}$ we get that $\xave (0)=(k+\frac
12)\Delta$ for some $k\in\integers$.
\end{proof}

\begin{remark}[Comparison with discrete-time consensus]
The results of ``practical consensus'' in Theorem~\ref{theor:ConvergenceStrip} and~\ref{theor:convergenceD} can be related with results about discrete-time consensus systems, as those in~\cite{PF-RC-FF-SZ:08} and~\cite{AK-TB-RS:07}. In~\cite{PF-RC-FF-SZ:08}, the authors consider
$x(t+1)=x(t)-\frac1{\dmax+\eps}L\qd(x(t))$, where $\dmax$ is the
largest in-degree in $\G$, and $\eps>0$.
Clearly $\D$ is the set of equilibria for this system, but its
limit behavior shows limit cycles which are not contained in the
closure of $\D$. However, for some example topologies (rings and complete graphs) the system can be proved to approach average consensus up to the quantizer precision.
The paper~\cite{AK-TB-RS:07} considers a discrete-state dynamics in which the agents communicate in randomly chosen pairs, and proves convergence to a set of the form
$\setdef{x\in \integer^N}{ x_i \in\{L,L+1\}, L \in \integer}.$
\end{remark}



\section{Chattering-free quantized dynamics}\label{sec:HysteresisQuant}
The analysis in Remark~\ref{rem:sliding-mode} has pointed out the
existence of sliding modes in the system $\dot x=-L {\tt q}(x)$, a
phenomenon which is not acceptable  in practical implementation.
In this section we discuss a different quantization scheme to
overcome this difficulty and we analyze the resulting system. 

\subsection{Hysteretic quantizer: hybrid model}

The different quantization scheme is based on a quantizer with
hysteresis which we introduce below. Let us consider the
multi-valued map $\qh$
defined as (Figure~\ref{fig:quantizer.hystersis})
\be\label{q.h} \qh(r)= \left\{ \ba{lrl}
j\Delta & -\frac{\Delta}{2}+j\Delta \le r < \frac{\Delta}{2}+j\Delta\\
j\Delta+\frac{\Delta}{2} &  j\Delta \le r < \Delta+j\Delta, & j\in
\Z \ea \right. \ee
for the scalar $r$, with the understanding that if $x\in\real^n$, $\qh(x)=(\qh(x_1),\ldots,\qh(x_n))^*$
\begin{figure}[htb]
\begin{center}
\includegraphics[width=.49\textwidth]{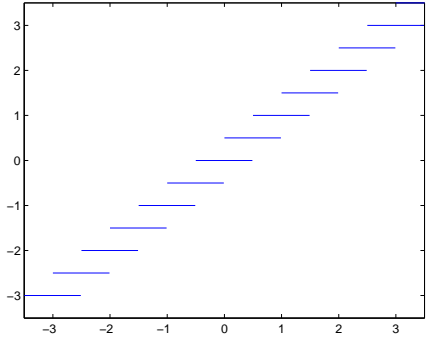}
\caption{The uniform quantizer with hysteresis $\qh(z)$, when $\Delta=1$.}\label{fig:quantizer.hystersis}
\end{center}
\end{figure}
The evolution of $\qh(x(t))$ as a function of $x(t)$ can be
described as follows. At time $t=0$, $\qh(x(0))= \qd (x(0))
$. Let, for the sake of notational simplicity, $\qh:=\qh(x(t))$, and $\qh^+:=\lim_{s\to t^+}\qh(x(s))$,
$x:=x(t)$.
If $\qh \in (\Delta \Z\cup (\Delta \Z+\Delta/2))^N$,
then
\be\label{switch.rule} \qh^+= \left\{\ba{cc}
\qh  +\frac{\Delta}{2} & 
\text{ if } x \ge \qh  +\frac{\Delta}{2}\\[2mm]
\qh  -\frac{\Delta}{2} & 
\text{ if } x\le \qh  -\frac{\Delta}{2}\\[2mm]
\qh  & 
\text{ if } \qh  -\frac{\Delta}{2}< x< \qh  +\frac{\Delta}{2}. \ea\right.
\ee

The rule (\ref{switch.rule}) is illustrated in Figure \ref{fig.switch}.
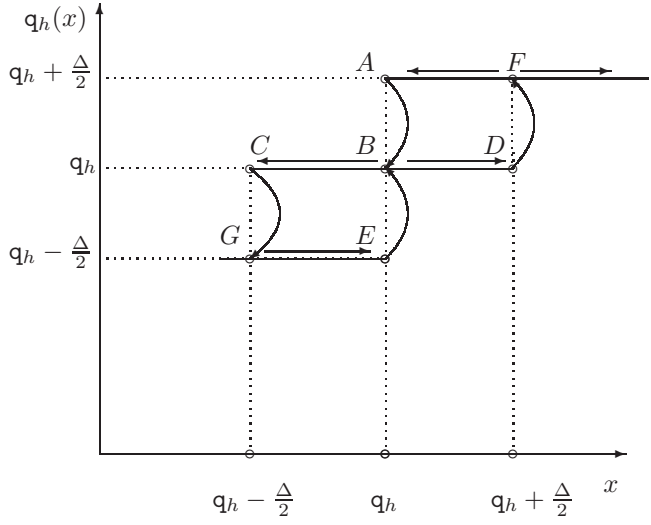
\begin{figure}
\setlength{\unitlength}{1mm}
\begin{center}
\scalebox{1}{
\begin{picture}(80,70)(0,0)
\put(77,5){$x$}
\put(10,10){\vector(1,0){70}}
\put(10,10){\vector(0,1){60}}
\put(0,67){${\tt q}_h(x)$}
%
%
\put(48,60){\circle{1}} \put(44,61){$A$}
%
%
%
\put(65,60){\circle{1}} \put(64,61){$F$}
%
%
\put(48,48){\circle{1}} \put(44,50){$B$}
%
%
\put(30,50){$C$}
%
%
\put(65,48){\circle{1}}
\put(61,50){$D$}
%
\multiput(65,50)(0,-1){40}{{\rule{.4pt}{.4pt}}}
\put(65,10){\circle{1}}
\put(62,3){${\tt q}_h+\frac{\Delta}{2}$}
%
%
%
\put(44,38){$E$}
\put(48,36){\circle{1}}
\put(48,10){\circle{1}}
%
%
\put(26,38){$G$}
\put(48,36){\circle{1}}
\put(48,10){\circle{1}}
%
%
\put(48,60){\line(1,0){35}}
\multiput(48,60)(-1,0){38}{{\rule{.4pt}{.4pt}}}
\put(-2,60){${\tt q}_h+\frac{\Delta}{2}$}
%
%
\put(48,36){\line(-1,0){22}}
\multiput(48,36)(-1,0){38}{{\rule{.4pt}{.4pt}}}
\put(-2,36){${\tt q}_h-\frac{\Delta}{2}$}
%
%
%
\multiput(48,60)(0,-1){50}{{\rule{.4pt}{.4pt}}}
\put(46,3){${\tt q}_h$}
%
%
\multiput(30,36)(0,-1){27}{{\rule{.4pt}{.4pt}}}
\put(30,36){\circle{1}}
\put(25,3){${\tt q}_h-\frac{\Delta}{2}$}
\put(30,10){\circle{1}}
%
%
\put(65,48){\line(-1,0){35}}
\multiput(30,48)(-1,0){20}{{\rule{.4pt}{.4pt}}}
\put(6,48){${\tt q}_h$}
\put(30,48){\circle{1}}
%
%
\put(48,48){\vector(-1,-1){.07}}\qbezier(48,60)(54,54)(48,48)
%
%
\put(48,48.25){\vector(-1,1){.07}}\qbezier(48,36)(54,41.875)(48,48.25)
%
%
%
\put(51,49){\vector(1,0){13}}
%
%
%
\put(47,49){\vector(-1,0){16}}
\multiput(64.8,47.93)(0,.92308){14}{{\rule{.4pt}{.4pt}}}
%
%
\put(64.8,60){\vector(-2,3){.07}}\qbezier(64.8,48)(70.8,54)(64.8,60)
%
%
\put(63,61){\vector(-1,0){12}} \put(66,61){\vector(1,0){12}}
%
\multiput(30,36)(0,1){12}{{\rule{.4pt}{.4pt}}}
\put(30,36){\vector(-4,-3){.07}}
\qbezier(30,48)(38,42)(30,36)
\put(32,37){\vector(1,0){14}}
\end{picture}
}
\caption{\label{fig.switch} The figure illustrates the evolution of
${\tt q}_h(x)$ as a function of $x$ as described in
(\ref{switch.rule}). Suppose initially that $\qh
-\frac{\Delta}{2}< x< \qh +\frac{\Delta}{2}$ and that
$\qh(x)=\qh$. As $x$ evolves, the point $(x,\qh(x))$ lies on the
segment $\overline{CD}$, and it may hit the point $C$ or the point
$D$, thus triggering a discrete transition. If the former occurs,
the discrete value takes a new value, namely $\qh^+=\qh-\frac{\Delta}{2}$,
the point $(x,\qh(x))$ jumps to $G$ and can then move towards $E$
or in the opposite direction. On the other hand, if $(x,\qh(x))$
hits $D$, then the quantization level takes the new value
$\qh+\frac{\Delta}{2}$, i.e.\ $\qh^+=\qh+\frac{\Delta}{2}$. In the
graph, this corresponds to the transition from point $D$ to point
$F$. The state $x$ can then further increase or decrease.
}
\end{center}
\end{figure}

Suppose that in the dynamics~\eqref{eq:StateQuant}, each agent
quantizes the information using $\qh$ rather than ${\tt q}$.
This leads to a system which can be better described and analyzed
using the formalism of  hybrid systems. To this end, we adopt the
notations in~\cite{RG-RS-AT:09}.
Let $q\in (\Delta \Z\cup (\Delta \Z+\Delta/2))^N$ be the discrete
state, $x\in \R ^N$ the continuous state and $X=\R^N \times
(\Delta \Z\cup (\Delta \Z+\Delta/2))^N$ the state space where the
system evolves. The continuous dynamics of the system are
described by
\be\label{continuous} \ba{rcl}
\dot x &=& -L q\\
\dot q &=& 0 \ea\ee
which are valid as far as the state $(x,q)$  belongs to the subset
of the state space:
\[
C=\{ (x,q)\in X:\forall\, i\in {\V}, -\frac{\Delta}{2}+q_i <
x_i < \frac{\Delta}{2}+q_i\}.
\]
If on the other hand $(x,q)$ belongs to the set
\[
D=\{ (x,q)\in X: \exists\, i\in {\V}, \;\text{such that}\; x_i\le
-\frac{\Delta}{2}+q_i\;\text{or}\; x_i\ge  \frac{\Delta}{2}+q_i\},
\]
then the following discrete update occurs: \be\label{discrete}
\ba{rcll}
x^+ &=& x\\
q_i^+ &=& \left\{\ba{ll}
q_i+\frac{\Delta}{2} &{\rm if\,}   x_i\ge \frac{\Delta}{2}+q_i\\[2mm]
q_i-\frac{\Delta}{2} & {\rm if\,}  x_i\le - \frac{\Delta}{2}+q_i\\[2mm]
q_i & \  {\rm otherwise}. \ea\right. \ea\ee
Observe that $C\cup D= X$ and $C\cap D= \emptyset$.
The equations~\eqref{continuous} and~\eqref{discrete} together
with the sets $C, D$ define the hybrid model associated with the
multi-agent system in the presence of the quantizers (\ref{q.h}).
In what follows we let $z=(x^\ast\;q^\ast)^\ast$ be the entire
state of the hybrid model and $f(z)$, $g(z)$ the maps on the
righthand side of~\eqref{continuous} and, respectively,~(\ref{discrete}). Hence, the system is concisely described as
\be\label{hs} \ba{rcll}
\dot z&=& f(z) & z\in C\\
z^+&=& g(z) & z\in D\;,\ea \ee
with an initial condition belonging to
\[
X_0= \{ (x,q)\in X: q_i-\frac{\Delta}{2}\le x_i <
q_i+\frac{\Delta}{2},\; \forall i\in {\V}\}.
\]

Note that initial conditions of the form $\left(x_0,\qd(x_0)\right),$ with $x_0\in\reals^N$, belong to this set.

We recall from \cite{RG-RS-AT:09} the notion of {\em hybrid time
domain} and {\em solution} for a hybrid system. A {\em hybrid time domain} is a subset of $\R_{\ge 0}\times \N$ which is the union of infinitely many intervals of the form $[t_j, t_{j+1}]\times
\{j\}$, where $0=t_0\le t_1\le t_2\le \ldots$, or of finitely many
such intervals with the last one possibly of the form $[t_j,
t_{j+1}]\times \{j\}$, $[t_j, t_{j+1})\times \{j\}$, or $[t_j,
+\infty)\times \{j\}$. Let $z(t,j)$ be a function defined on a
hybrid time domain ${\rm dom} z$ such that for each fixed $j$,
$t\mapsto z(t,j)$ is a locally absolutely continuous function on the
interval $I_j=\{t:(t,j)\in {\rm dom} z \}$. The function $z(t,j)$ is a {\em solution}
to the hybrid system (\ref{hs}) if $z(0,0)\in C\cup D$ and the
following conditions are satisfied:
\begin{itemize}
\item For each $j$ such that $I_j$ has non-empty interior,
\[
\ba{rcll}
\dot z(t,j) &=&  f(z(t,j)) & \,{\rm for \, a.e.}\, t\in I_j\\
z(t,j) &\in & C & \, {\rm for\, all}\,  t\in [\min I_j, \sup I_j) \ea
\]

\item For each $(t,j)\in {\rm dom} z$ such that $(t,j+1)\in {\rm dom}
z$,
\[
\ba{rcl}
z(t,j+1) &=&  g(z(t,j))\\
z(t,j) &\in & D. \ea
\]

\end{itemize}
The solution $z$ is {\em nontrivial} if ${\rm dom} z$ contains at
least another point different from $(0,0)$ and is {\em complete}
if ${\rm dom} z$ is unbounded.

\subsection{Hybrid model analysis}
In this subsection we detail the analysis of system~\eqref{hs}. After proving basic properties about existence,
uniqueness and completeness of solutions, we verify that
chattering can not occur. We also compute the equilibria of the
system, and present a convergence result.

\begin{lemma}[Basic properties of solutions]\label{lemma:Basic-hyst}
For each $z(0,0)\in X_0$, there exists a non-trivial solution
$z(t,j)$ to~(\ref{hs}), and such solution is {\em forward
unique} and {\em complete}.
Moreover, every interval $I_j=\{t:(t,j)\in {\rm dom} z \}$,
possibly with the exception of $I_0$,  has a non-void interior.
\end{lemma}
\begin{proof}
There are two cases: $z(0,0)\in D$ or $z(0,0)\in C$. If $z(0,0)\in D$, then
$z(0,1)\in C$. Indeed, let  $i\in {\V}$ be any
index such that $x_i(0,0)= \frac{\Delta}{2}+q_i(0,0)$. Then
$x_i(0,1)=x_i(0,0)$ and $q_i(0,1)=\frac{\Delta}{2}+q_i(0,0)$, from
which $x_i(0,1)=q_i(0,1)$. Similarly, for any index $k$ for which
$x_k(0,0)= -\frac{\Delta}{2}+q_k(0,0)$, we have $x_k(0,1)=x_k(0,0)$ and
$q_k(0,1)=-\frac{\Delta}{2}+q_k(0,0)$.
For all the remaining indices $\ell$, $q_\ell(0,0)-\frac{\Delta}{2}<
x_\ell(0,0) < q_\ell(0,0)+\frac{\Delta}{2}$ and then
$q_\ell(0,1)=q_\ell(0,0)$ and $x_\ell(0,1)=x_\ell(0,0)$.
Then $z(0,1)\in C$
and the discrete transition stops in one step.
%
Since the righthand side of (\ref{continuous}) is constant, the
solution $z(t,j)$ starting either from $z(0,1)$ (if $z(0,0)\in D$)
or from $z(0,0)$ (if $z(0,0)\in C$) exists for $t$
sufficiently close to $0$ and is unique.
As a matter of fact, one can compute explicitly the solution  $z(t,j)$ on the
interval $I_j$, with $j=0$  (if $z(0,0)\in C$) or $j=1$  (if $z(0,0)\in D$).
Define the auxiliary Cauchy problem $\dot \zeta(t)=-L q(0,j)$, $\zeta(0)=x(0,j)$,
where by a slight abuse of notation we are denoting the time variable in the auxiliary system
by the same symbol $t$ which appears in the hybrid time $(t,j)$.
The problem admits the unique solution $\zeta(t)$ which is defined for all $t \ge 0$.
Let $\bar t$ be the minimal time greater than zero for which there exists an index $k\in {\V}$
such that $\zeta_k(\bar t\,)= -\frac{\Delta}{2}+q_k(0,j)$ or
$\zeta_k(\bar t\,)= \frac{\Delta}{2}+q_k(0,j)$. Set $I_j=[0,\bar t\,]$ if a finite $\bar t$ exists,
otherwise set $I_j=[0,+\infty)$. Finally let $z(t,j)=(\zeta(t)^*\; q(0,j)^*)^*\;$ for all $t\in I_j$. Then,
by construction, $z(t,j)$ satisfies the condition (\ref{continuous}).
If $\sup I_j=+\infty$, we have concluded the argument. If instead $\sup I_j<+\infty$,
$z(t,j)$ satisfies the differential equation $\dot z(t,j) =  f(z(t,j))=-L q(0,j)$ for all $t\in I_j$
except at $t=\sup I_j$; but $z(\sup I_j, j)\in D$ and at time $t=\sup I_j$ a discrete transition occurs.
After this transition, the solution can be uniquely extended starting from $z(\sup I_j, j+1)\in C$,
repeating the argument which have been just discussed.
Then, $z(t,j)$ is forward unique because it results from the concatenation
of unique solutions to~(\ref{continuous}) and (\ref{discrete}).

Let us show that the solution is complete. By contradiction, ${\rm dom} z$ is
bounded, that is ${\rm dom} z$ is the union of finitely many
intervals of the form $[t_j, t_{j+1}]\times \{j\}$, with the last
interval either of the form $[t_j, t_{j+1}]\times \{j\}$ or
$[t_j,t_{j+1})\times \{j\}$ and $t_{j+1}<+\infty$. This yields a
contradiction, because in view of the form of the last interval,
no more discrete transitions take place and by (\ref{continuous})
the continuous evolution exists for all $t$.

Finally, the interval $I_j=\{t:(t,j)\in {\rm dom} z \}$, with $j$
such that $z(0,j)\in C$, has a non-void interior because the
solution to (\ref{hs}) from $z(0,j)$ flows continuously with
bounded velocity. If $\sup I_j=+\infty$, then $(I_j,j)$ is the
last interval  of ${\rm dom}\; z$ and the proof ends. Otherwise,
if $\sup I_j<+\infty$, then by definition $z(\sup I_j,j)\in D$ and
there is an index $i\in {\V}$ such that either
$x_i(\sup I_j,j)= \frac{\Delta}{2}+q_i(\sup I_j,j)$
or $x_i(\sup I_j,j)=-\frac{\Delta}{2}+q_i(\sup I_j,j)$. Then, as
in the first part of the proof, one proves that $z(\sup
I_j,j+1)\in C$ and again $I_{j+1}$ has a non-empty interior
because the system flows continuously with bounded velocity. The
thesis descends by induction.
\end{proof}

The following lemma proves boundedness of solutions, a property which is useful in order to prove that chattering does not occur.

\begin{lemma}[Bounded solutions]\label{l1}
For every initial condition $z(0,0)\in X_0$, the solution $z(t,j)$ satisfies
\[\ba{c}
-\frac{\Delta}{2}+\min\{q_1(0,0),\ldots,q_N(0,0)\}\le q_i(t,j)\le
\max\{q_1(0,0),\ldots,q_N(0,0)\}+\frac{\Delta}{2}\\[2mm]
-\Delta+\min\{q_1(0,0),\ldots,q_N(0,0)\}\le x_i(t,j)\le
\max\{q_1(0,0),\ldots,q_N(0,0)\}+\Delta \;, \ea\] for all $i\in
{\V}$ and for all $(t,j)\in {\rm dom} z$.
\end{lemma}
\begin{proof}
Consider the function $U(z)=\max\{q_1,\ldots,q_N\}$ and compute it
along the solution to (\ref{hs}). During continuous evolutions,
i.e.\ if $z(t,j)\in C$, $U(z(t))$ remains constant since, by the
second equation in (\ref{continuous}), no component of $q(t,j)$
changes its value. If $z(t,j)\in D$ and $(t,j)\ne (0,0)$, then
$U(z(t,j+1))\le U(z(t,j))$. In fact, suppose this is not true.
Then, since $q(t,j)\in (\Delta \Z\cup (\Delta \Z+\Delta/2))^N$,
and at the transition each component $q_i$ either remains the same
or changes its value by $\pm \Delta/2$, i.e.\
$q_i(t,j+1)=q_i(t,j)$ or $q_i(t,j+1)=q_i(t,j)\pm \Delta/2$, the
only possibility for having $U(z(t,j+1))> U(z(t,j))$ is that one
of the agents, say $k$, for which its discrete state $q_k(t,j)$
equals $\max\{q_1(t,j),\ldots,q_N(t,j)\}$, increases its value of
$+ \Delta/2$ . But for this to occur, it must be true
that during the continuous evolution which preceded the transition,
the continuous state $x_k(t,j)$ has increased its value
until it reached the threshold
$\frac{\Delta}{2}+q_k(t,j)=x_k(t,j)$. This is a contradiction
because, since $q_k(t,j)=\max\{q_1(t,j),\ldots,q_N(t,j)\}$, in the
equation $\dot x_k(t,j)=\sum_{j=1}^N A_{ij}(q_j(t,j)-q_k(t,j))$
the righthand side is non-positive. Hence, we conclude that
$U(z(t,j))\le U(z(0,k))$ for all $(t,j)\in {\rm dom}z$ such that
$t\ge 0$ and $j\ge k$, where $k\in \{0,1\}$ is the smallest
integer such that $z(0,k)\in C$.
In case $k=0$, then $U(z(t,j))\le U(z(0,0))$. In case $k=1$, then
$U(z(0,1))\le U(z(0,0))+\Delta/2$, and $U(z(t,j))\le U(z(0,1))\le
U(z(0,0))+\Delta/2$. Similarly one can prove that
$W(z)=\min\{q_1,\ldots,q_N\}$ satisfies $W(z(t,j))\ge W(z(0,0))-\Delta/2$.\\
The conclusion on the continuous state follows easily since
$z(0,0) \in X_0$ implies that for each $(t,j)\in {\rm dom}z$, for
each $i\in {\V}$, $q_i(t,j)-\frac{\Delta}{2}\le x_i(t,j) \le
q_i(t,j)+\frac{\Delta}{2}$.
\end{proof}

\bigskip

The last part of the thesis of Lemma \ref{lemma:Basic-hyst} implies that  the set $\{t\in \R_{\ge
0}|z(t,j)\in D\}$ has a zero measure. Moreover, a finer analysis
can be performed. Let us consider the switching times of a
solution $z(t,j)$ to the hybrid system originating from $z(0,0)\in
X_0$, i.e.\ the times $t_j$ such that $(t_j,j)\in {\rm dom} z$
implies $(t_j,j+1)\in {\rm dom} z$.
The set of switching times is {\em locally finite} if, for any
compact subset of  $\R_{\ge 0}$, there is only a finite number of
switching times which belong to that compact subset.
\begin{proposition}[No chattering]\label{p.no.chattering}
For every initial condition $z(0,0)\in X_0$, the set of switching
times of the solution $z(t,j)$ is locally finite.
\end{proposition}

\begin{proof}
First of all we remark that for each agent $j$, each time a
switching occurs, there must elapse an interval of time of length
at least
$\frac{\Delta/2}{||L||_\infty (||q(0,0)||_\infty+\Delta/2)}$
before the agent $i$ switches again: that is, for each
agent the inter-switching intervals have lengths which are bounded
from below.  Indeed, the agent's state $x_i(t,j)$
evolves with speed not larger than $||L||_\infty
(||q(0,0)||_\infty+\Delta/2)$ (see Lemma~\ref{l1}) and, after each
switching, it has to cover a distance of length at least
$\Delta/2$ before it fulfills a switching condition again.

Let us now suppose  that there exists a solution $z(t,j)$ and a
sequence of switching times $t_j$, with $(t_j,j)\in {\rm dom}z$,
such that $j\to \infty$  and
\be\label{zeno} \lim_{j\to +\infty}\dst\sum_{k=0}^j
(t_{k+1}-t_k)=T<+\infty\;. \ee
Because the number of switches is infinite and the number of
agents is finite, then there is at least an index $i\in {\V}$
for which $q_i$ switches infinitely many times. We can then extract the
sequence $t_{i_k}$ of times at which $q_i$ changes its value. Clearly,
$$
\lim_{j\to +\infty}\dst\sum_{k=0}^j
(t_{i_{k+1}}-t_{i_k})<+\infty\;.
$$
On the other hand, thanks to the previous argument, one has that
$$t_{i_{k+1}}-t_{i_k}\geq \frac{\Delta/2}{||L||_\infty
(||q(0,0)||_\infty+\Delta/2)},$$ and then
$$
\lim_{j\to +\infty}\dst\sum_{k=0}^j
(t_{i_{k+1}}-t_{i_k})=+\infty\;,
$$
i.e.\ a contradiction.
\end{proof}

\begin{remark}[On the ``continuous component" of the hybrid time domain]
Combining Lemma \ref{lemma:Basic-hyst} and Proposition
\ref{p.no.chattering} it is possible to conclude that
the continuous component of the hybrid time domain
 is unbounded.
Namely
${\rm dom} z$  is
either the union of infinitely many intervals of the form $[t_j,
t_{j+1}]\times \{j\}$, with $0=t_0\le t_1\le \ldots$, or the union
of finitely many of such intervals, the last one being $[t_j,
+\infty)\cup \{j\}$. Moreover, in the former case, necessarily
$\cup_{j\in \N} [t_j, t_{j+1}]=\R_{\ge 0}$ by Proposition
\ref{p.no.chattering}.
\end{remark}

\begin{remark}[Data rate]
In the case of quantizers with hysteresis, thanks to Lemma~\ref{l1} and Proposition~\ref{p.no.chattering}, one can find an upper bound on the the data rate at which each agent transmits information to its neighbors.
In fact, the proof of Proposition~\ref{p.no.chattering} tells us that the information regarding the state is transmitted at most every
$T:=\frac{\Delta/2}{||L||_\infty (||q(0,0)||_\infty+\Delta/2)}$
units of time. The
information consists of packets of bits which encode the quantized
state of the agents. In view of Lemma~\ref{l1}, the number of
quantization levels employed by each agent is
$4\frac{||q(0,0)||_\infty+\Delta/2}{\Delta/2}+1$. Hence, to encode
these quantization levels,
$B:=\lceil \log_2 \left(8\frac{||q(0,0)||_\infty}{\Delta}+5\right)\rceil$
bits are needed.
We conclude that each agent transmits information at a data rate
which is not larger than
\[
\dst\frac{B}{T}= \left\lceil \log_2 \left(8
\frac{||q(0,0)||_\infty}{\Delta}+5\right)\right\rceil
\left(2\frac{||q(0,0)||_\infty}{\Delta}+1\right) ||L||_\infty.
\]
\end{remark}

\medskip

We can also prove that the solution to (\ref{hs}) preserves the average of the continuous
states.
\begin{lemma}[Average preservation]\label{l.preserve}
For each $z(0,0)\in X_0$, the solution $z(t,j)$ to (\ref{hs})
is such that $N^{-1}\1^\ast x(t,j)=N^{-1}\1^\ast
x(0,0)$, for all $(t,j)\in {\rm dom} z$.
\end{lemma}

\begin{proof}
By (\ref{discrete}), for all $(t,j)\in {\rm dom} z$ such that
$(t,j+1)\in {\rm dom} z$, $x(t,j)= x(t,j+1)$, i.e.\ during
discrete transitions the value of $x$ does not change and
therefore the average is trivially preserved. For $z(t,j)\in C$,
the continuous state $x(t,j)$ satisfies $N^{-1}\1^\ast
\dot x(t,j)= -N^{-1}\1^\ast L q(t,j)$ for all $t\in I_j$.
Since the graph is weight-balanced, then $\1^\ast
L=\mathbf{0}^\ast$, and therefore the average is preserved also
during continuous flow. This leads to the thesis.
\end{proof}

\bigskip
Consider now the set of equilibria of the system, i.e.\ the set of
states such that all the evolutions originating from those remain
in the same state. This set is characterized in the following
statement.
\begin{lemma}[Equilibria]
The set of equilibria for the system (\ref{hs}) is ${\Eq}={\Eq}_1\cup {\Eq}_2$, with
\[\ba{l}
{\Eq}_1= \{(x,q)\in X \,:\, \exists k\in \Z\;\textup{ s.t.}\;
q_i=k\Delta,\, \,\textup{and }\,
-\frac{\Delta}{2}+q_i<x_i<\frac{\Delta}{2}+q_i,\, \forall\, i\in
{\V}\}\\
{\Eq}_2= \{(x,q)\in X \,:\, \exists k\in \Z\,\;\textup{ s.t.}\;\,
q_i=k\Delta+\frac{\Delta}{2},\, \,\;\textup{and }\,
-\frac{\Delta}{2}+q_i<x_i<\frac{\Delta}{2}+q_i,\, \forall\, i\in
{\V}\}\;. \ea\]
\end{lemma}

\begin{proof}
Suppose that the initial condition $z(0,0)$ belongs to ${\Eq}_1$. Then, since $q_i(0,0)=k\Delta$ and
$-\frac{\Delta}{2}+q_i(0,0)<x_i(0,0)<\frac{\Delta}{2}+q_i(0,0)$
for all $i\in {\V}$, $z(0,0)\in C$,  no discrete transition is
triggered, and moreover $-L q(0,0)=\mathbf{0}$, i.e.\ during the
continuous flow the state remains unchanged. The same can be argued if the
initial condition $z(0,0)$ belongs to
${\Eq}_2$.

Suppose now that $z(0,0)\in X$ is an equilibrium point, i.e.\
$z(t,j)=z(0,0)$ for all $(t,j)\in {\rm dom}z$. Then $z(0,0)\not
\in D$, because otherwise a discrete update via (\ref{discrete})
would occur leading to $q(0,1)\ne q(0,0)$, a contradiction. Since
$X=C\cup D$, then $z(0,0)\in C$. In order to have
$x(t,0)=x(0,0)$ for all $t\in I_0$, it must necessarily be true
that $-L q(0,0)=\mathbf{0}$, i.e.\ $q(0,0)\in {\rm span}
\1\cap (\Delta \Z\cup (\Delta \Z+\Delta/2))^N$ or,
equivalently, that for each $i\in {\V}$, either
$q_i(0,0)=k\Delta+\frac{\Delta}{2}$ or $q_i(0,0)=k\Delta$ for some
$k\in \mathbb{Z}$. Since $z(0,0)\in C$, then
necessarily
$-\frac{\Delta}{2}+q_i(0,0)<x_i(0,0)<\frac{\Delta}{2}+q_i(0,0)$,
for all $i\in {\V}$. We conclude that $z(0,0)\in
{\Eq}_1\cup {\Eq}_2$.
\end{proof}

\medskip

Differently from the case of quantizers with no hysteresis, the
set of equilibria ${\Eq}$ may not be globally attractive for
the solutions to~(\ref{hs}) as the following example
shows.
\begin{example}[Finite-time limit cycle]\label{ex:LimitCycle}
Consider the system~(\ref{continuous})-(\ref{discrete}) with $N=2$
and where $\dot x=-L q$ is
\be\label{continuous.N=2} \ba{rcl}
\dot x_1 &=& q_2-q_1\\
\dot x_2 &=& q_1-q_2\;. \ea \ee
Let $q_1 \in \Delta \integers$ be fixed and consider the initial
condition
$$z_1(0,0)=(q_1-\frac{\Delta}{4},q_1),\  z_2(0,0)=(q_1+\frac{3\Delta}{4}
,q_1+\Delta).$$
The continuous dynamics is given by
$$
\ba{rcl}
\dot x_1 &=& {\Delta}\\
\dot x_2 &=& -{\Delta}\; \ea $$ so that
$$
\ba{rcl}
 x_1 (t) &=& q_1-\frac{\Delta}{4}+{\Delta}t\\
 x_2 (t)&=& q_1+\frac{3\Delta}{4}-{\Delta}t\; \ea $$
on the interval $[0,\frac{1}{4}]\times \{ 0 \}$.
Indeed, at time $t=\frac{1}{4}$ we have $x_1 (\frac 14) = q_1,\
 x_2 (\frac 14)= q_1+\frac{\Delta}{2}$, so that a discrete transition is triggered.
It holds
$$z_1(\frac 14,1)=(q_1,q_1), \qquad  z_2(\frac 14,1)=(q_1+\frac{\Delta}{2}
,q_1+\frac{\Delta}{2}).$$
The continuous dynamics is then given by
$$
\ba{rcl}
\dot x_1 &=& \frac{\Delta}{2}\\
\dot x_2 &=& -\frac{\Delta}{2}\; \ea $$ so that
$$
\ba{rcl}
 x_1 (t) &=& q_1+\frac{\Delta}{2}(t-\frac 14)\\[2mm]
 x_2 (t)&=& q_1+\frac{\Delta}{2}-\frac{\Delta}{2}(t-\frac 14)\; \ea $$
on the interval $[\frac 14,\frac 54]\times \{ 1 \}$.
At  time $t=\frac 54$ we have $x_1 (\frac 54) = q_1+\frac{\Delta}{2},\
 x_2 (\frac 54)= q_1$, so that a new discrete transition is triggered. We
 have:
$$z_1(\frac 54,2)=(q_1+\frac{\Delta}{2},q_1+\frac{\Delta}{2}),\ \qquad  z_2(\frac 54,2)=(q_1
,q_1).$$
In the next interval the continuous dynamics is given by
$$
\ba{rcl}
\dot x_1 &=& -\frac{\Delta}{2}\\
\dot x_2 &=& \frac{\Delta}{2}\; \ea $$ so that
$$
\ba{rcl}
 x_1 (t) &=& q_1+\frac{\Delta}{2}- \frac{\Delta}{2}(t-\frac 54)\\[2mm]
 x_2 (t)&=& q_1+\frac{\Delta}{2}(t-\frac 54).\; \ea $$
At  time $t=\frac 94$ we have $x_1 (\frac 94) = q_1=x_1(\frac 14)$ and $x_2 (\frac 94)= q_1+\frac{\Delta}{2}=x_2(\frac 14)$, so that a new discrete transition is triggered. Hence, $z(\frac94,3)=z(\frac14,1)$ and the evolution has entered a cycle.

The conclusion is that {\em there exist trajectories of system~\eqref{continuous}-\eqref{discrete} which converge to periodic trajectories in finite time}, and in particular do not converge to the set of equilibria. In this bidimensional example, the exhibited periodic trajectory lies in the closure of the set of equilibria.  \hfill\QED
\end{example}
It should be noted that finite-time convergence of trajectories implies that solutions to~\eqref{hs} are not {\em backward} unique, whereas forward uniqueness has been proved in Lemma~\ref{lemma:Basic-hyst}.

In view of the existence of limit cycles, we can prove a convergence result analogous to Corollary~\ref{c9}.
\begin{theorem}[Finite-time convergence]\label{th.quantized}
Consider the system~(\ref{continuous})-(\ref{discrete}). Then, for
any $\eps\in (0,1)$, there exists $T(\varepsilon)>0$
such that the solution $z(t,j)$ to~(\ref{continuous})-(\ref{discrete}) satisfies
\[
||x(t,j)-\frac{\1\1^\ast}{N}x(0,0)||\le
\dst\frac{1}{1-\varepsilon}\dst\frac{||L||}{\lambda_2(\sym(L))} \dst\frac{\Delta}{2}\sqrt{N}
\]
for all $(t,j)\in{\rm dom} z$ such that $t\ge T(\varepsilon)$.
\end{theorem}

\begin{proof}
The proof follows the lines of those for Theorem~\ref{theor:ConvergenceStrip} and Corollary~\ref{c9}.
Define $y(t,j)=\Omega
x(t,j)=(I-\frac{\1\1^\ast}{N})x(t,j)$ and observe that for $z(t,j)\in C$, $y(t,j)$ satisfies
$$
\dot y(t,j)=(I-\frac{\1\1^\ast}{N})\dot x(t,j)=-(I-\frac{\1\1^\ast}{N})L q(t,j)= -L q(t,j),$$
while at each switching time $t_i$,  $y(t_i,i+1)=y(t_i,i)$. Given
$V(y)=\frac{y^\ast y}{2}$, we investigate $\frac{d V(y(t,j))}{dt}=\nabla V(y(t,j)) \dot y(t,j)$ when $z(t,j)\in C$. By Lemma~\ref{lemma:t1}, we obtain:
\[\ba{rcl}
\nabla V(y)(-L q)
&=& \nabla V(y)(-L(x+ q-x))\\
&= & -\nabla V(y)Lx- \nabla V(y)L(q-x)\\
&= & -\nabla V(y)L y- \nabla V(y)L(q-x)\\
& \le  & -\lambda_2(\sym(L))
||y||^2 + ||y||\,||L||\dst\frac{\Delta}{2} \sqrt{N}\\[2mm]
&= & -\lambda_2(\sym(L))||y|| \left(||y|| -
\dst\frac{||L||}{\lambda_2(\sym(L))} \dst\frac{\Delta}{2}
\sqrt{N}\right)\;, \ea\]
where in the inequality we exploit the fact that $z\in C$ implies
$q_i-\frac{\Delta}{2}\le x_i\le q_i+\frac{\Delta}{2}$ and hence
$|x_i-q_i|\le \frac{\Delta}{2}$ for all $i\in {\V}$.
If \be\label{yy} ||y|| > \dst\frac{1}{1-\varepsilon}
\dst\frac{||L||}{\lambda_2(\sym(L))} \dst\frac{\Delta}{2}
\sqrt{N}\;, \ee then
\[
\nabla V(y)(-Lq)\le -\varepsilon\lambda_2(\sym(L))||y||^2=-2\varepsilon\lambda_2(\sym(L))V(y)\;.
\]
Hence, for all  $y(t,j)\in C$ which satisfy~(\ref{yy}), it holds
$$\frac{d V(y(t,j))}{dt}\le -2\varepsilon\lambda_2(\sym(L))V(y(t,j)).$$ At each  switching time  $(t_i,i)$, on the
other hand, we have $V(y(t_i,i+1))=V(y(t_i,i))$. As a consequence,
denoted by $t_0=0, t_1, t_2, \ldots, t_j$ the switching times
which precede $t$, we have
\[\ba{rcl}
V(y(t,j))& \le & {\rm e}^{-2\varepsilon\lambda_2(\sym(L))(t-t_j)} V(y(t_j,j))\\
& \le & {\rm e}^{-2\varepsilon\lambda_2(\sym(L))(t-t_j)} V(y(t_j,j-1))\\
& \le & {\rm e}^{-2\varepsilon\lambda_2(\sym(L))(t-t_j)} {\rm e}^{-2\varepsilon\lambda_2(\sym(L))(t_j-t_{j-1})}
V(y(t_{j-1}, j-1)),\\
&\vdots&\\
&\le & {\rm e}^{-2\varepsilon\lambda_2(\sym(L))(t-t_j)} {\rm
e}^{-2\varepsilon\lambda_2(\sym(L))(t_j-t_{j-1})} \ldots {\rm
e}^{-2\varepsilon\lambda_2(\sym(L))(t_1-t_{0})} V(y(t_{0},0)),
\ea\]
which implies
\[
V(y(t,j))\le {\rm e}^{-2\varepsilon\lambda_2(\sym(L)(t-t_0)} V(y(t_0,0))
\]
and therefore
\[
||y(t,j)||\le {\rm e}^{-\varepsilon\lambda_2(\sym(L))(t-t_0)}
||y(t_0,0)||\;.
\]
From the latter we conclude that, if at time $(t_0,0)$,
condition~(\ref{yy}) is satisfied, then after at most
\begin{equation}\label{eq:Teps-hyst}
T(\varepsilon)=\max \left\{ 0, \dst\frac{-1}{\varepsilon \lambda_2(\sym(L))}
\ln \left( \dst\frac{1}{1-\varepsilon}
\dst\frac{||L||}{\lambda_2(\sym(L))} \dst\frac{\Delta}{2}
\dst\frac{\sqrt{N}}{||y(t_0,0)||} \right) \right\}
\end{equation}
units of time, $||y(t,j)||$, with $t\ge T(\varepsilon)$, satisfies
\[
||y|| < \dst\frac{1}{1-\varepsilon}
\dst\frac{||L||}{\lambda_2(\sym(L))} \dst\frac{\Delta}{2}
\sqrt{N}\;,
\]
and continue  to do so from that time on. The thesis then follows,
recalling the definition of $y(t,j)$ and that thanks to Lemma~\ref{l.preserve}, $\1^\ast
x(t,j)=\1^\ast x(0,0)$ for all $(t,j)$.
\end{proof}

\bigskip

One might wonder whether it is possible to make a stronger claim
of convergence to a set which does not depend on the network
topology, for instance to the closure of the set of equilibria.
The answer seems to be negative in general, as we show in the next
subsection for systems with more than two agents  by means of
simulation examples: limit cycles are inherent to the hysteretic
dynamics. An interesting open question, which we touch upon
commenting the simulations, is to relate these limit cycles, and
in particular their size, to the graph topology.

\begin{remark}[Convergence rate]\label{rem.conv}
From the proofs of  Corollary~\ref{c9} and of
Theorem~\ref{th.quantized} we can see that, away from a strip containing the equilibria, the rate of convergence can be bounded by the same
quantity for both algorithms. Actually, the rate is proportional
to $\lambda_2(\sym(L)),$ which is known to be the rate of convergence of the
non-quantized linear consensus dynamics. In other
words, our results suggest that the use of quantized measurements, with or without
hysteresis, does not affect the rate of convergence of the
consensus dynamics.
\end{remark}

\subsection{Simulations}\label{sec:Simulations}
In this section, we collect a few simulation
examples, focusing on the effects of hysteresis and on the
convergence properties of~\eqref{hs}.
We performed our simulations using a Matlab program which implements the hybrid
system~\eqref{hs}, and solves systems~\eqref{eq:qAlgo} and~\eqref{continuous} by an explicit Euler scheme with time step $\delta t=0.005$. In the set of simulations we show here, we assume that $N=10$, $A_{ij}\in\{0,1\}$ for every $i,j\in\V$, and $\Delta=0.05$, and that the
initial condition is $x(0)=[0.91728,0.26898,0.76538,0.18858,0.28738,0.09098,0.57608,0.68328,0.54648,0.42558]^*$.
\begin{figure}[htb]
\subfigure[]{
\includegraphics[width=.49\textwidth]{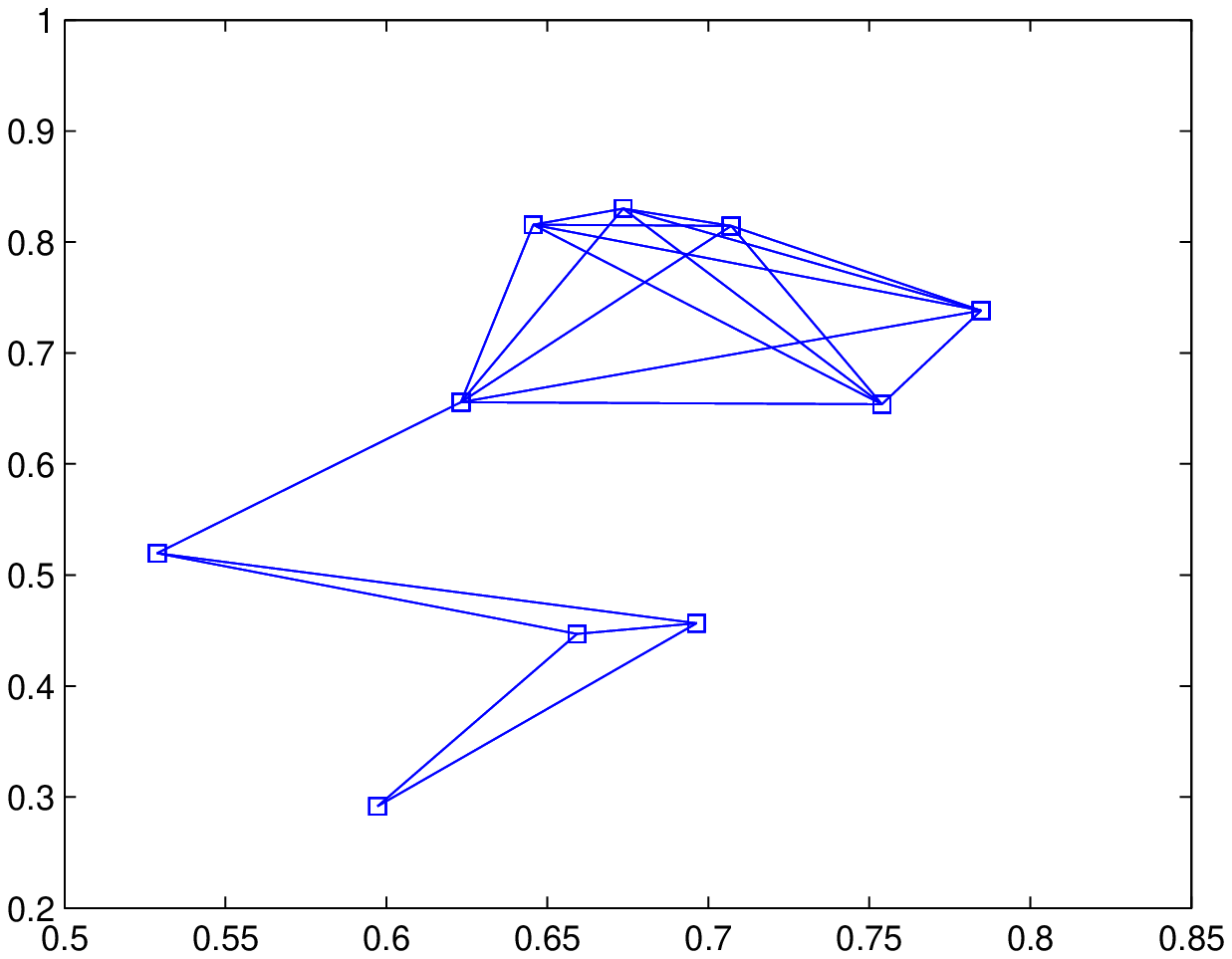}
\label{fig:RGGex-cluster}}
\subfigure[]{
\includegraphics[width=.49\textwidth]{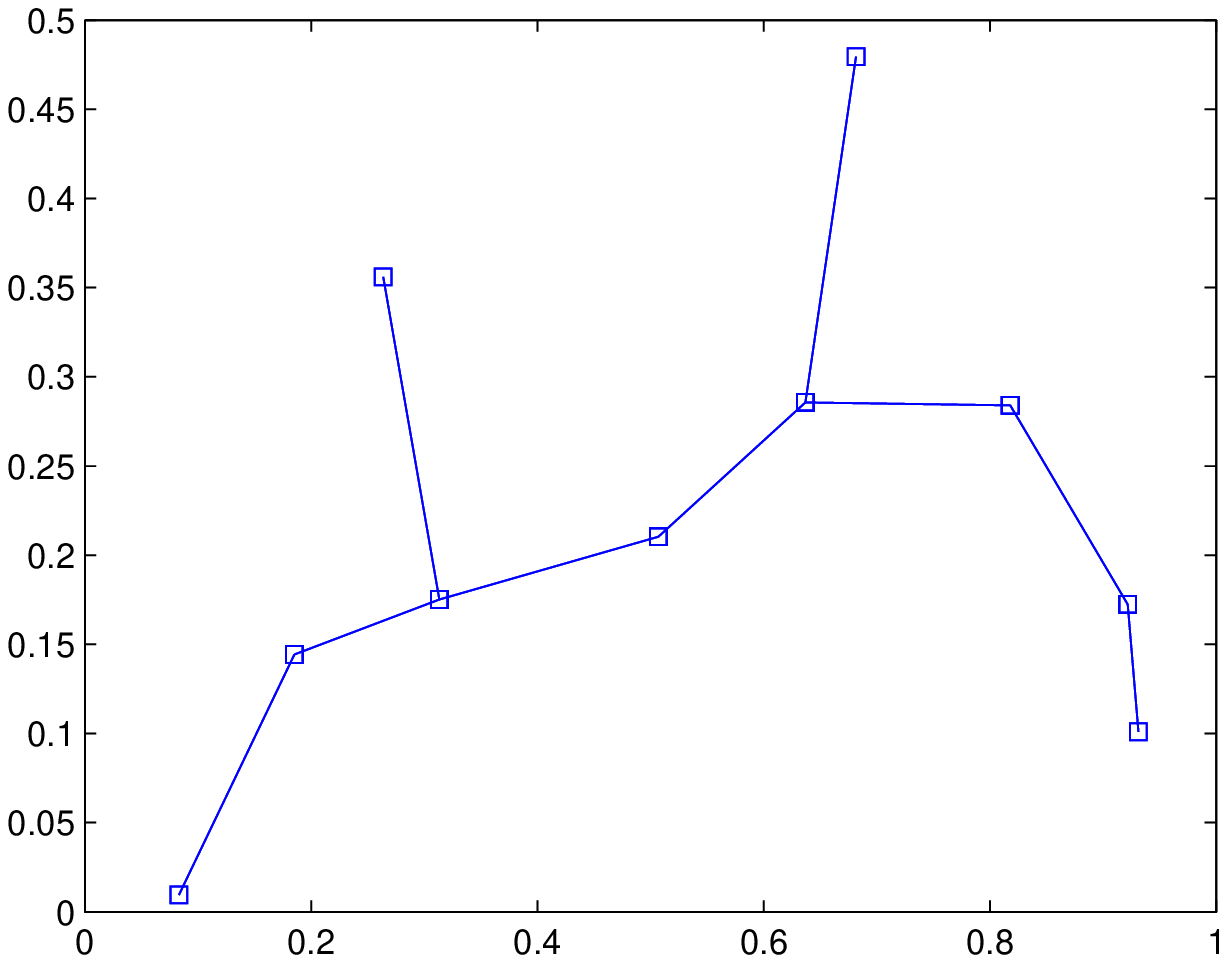}
\label{fig:RGGex-tree}}
\caption{Two sample graphs, obtained as realizations of a random
geometric graph~\cite{MP:03} in the unit square $[0,1]^2$ with
connectivity radius equal to $0.2$.} \label{fig:RGGex}
\end{figure}
Let us first consider the two networks depicted in
Figure~\ref{fig:RGGex}. Figures~\ref{fig:RGGsolutions1}
and~\ref{fig:RGGsolutions2} show for these networks the computed
evolutions of~\eqref{eq:qAlgo} and~\eqref{hs}, in terms of the
quantized states $\qd(x(t))$ and $q(t,j)$.
Notice that solutions to~\eqref{eq:qAlgo} suffer from chattering. Indeed,
there are subgraphs of the graph in Figure~\ref{fig:RGGex-tree} which coincide with the graph discussed in Example~\ref{ex1}. Hence, the chattering observed in the
simulations is due to a sliding mode similar to the one described
in the example and to the numerical implementation of the
algorithm.
Instead, solutions to~\eqref{hs} show no chattering: as expected,
hysteresis prevents chattering, and may affect the convergence
rate near the equilibria. The latter statement is not in contrast with
Remark~\ref{rem.conv}. In both the examples above, the quantized
states of system~\eqref{hs} converge to consensus.
\begin{figure}[htb]
\includegraphics[width=.49\textwidth]{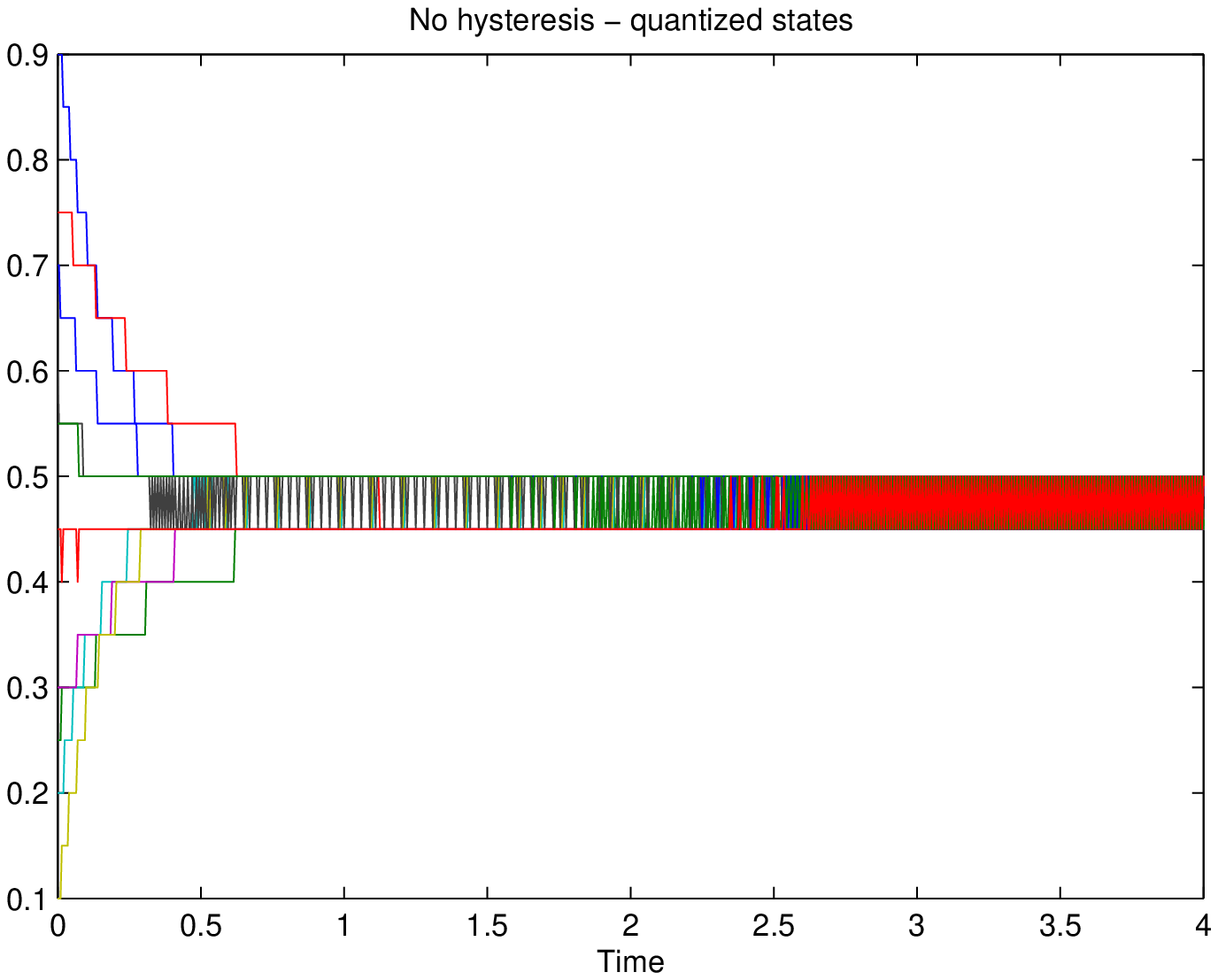}
\includegraphics[width=.49\textwidth]{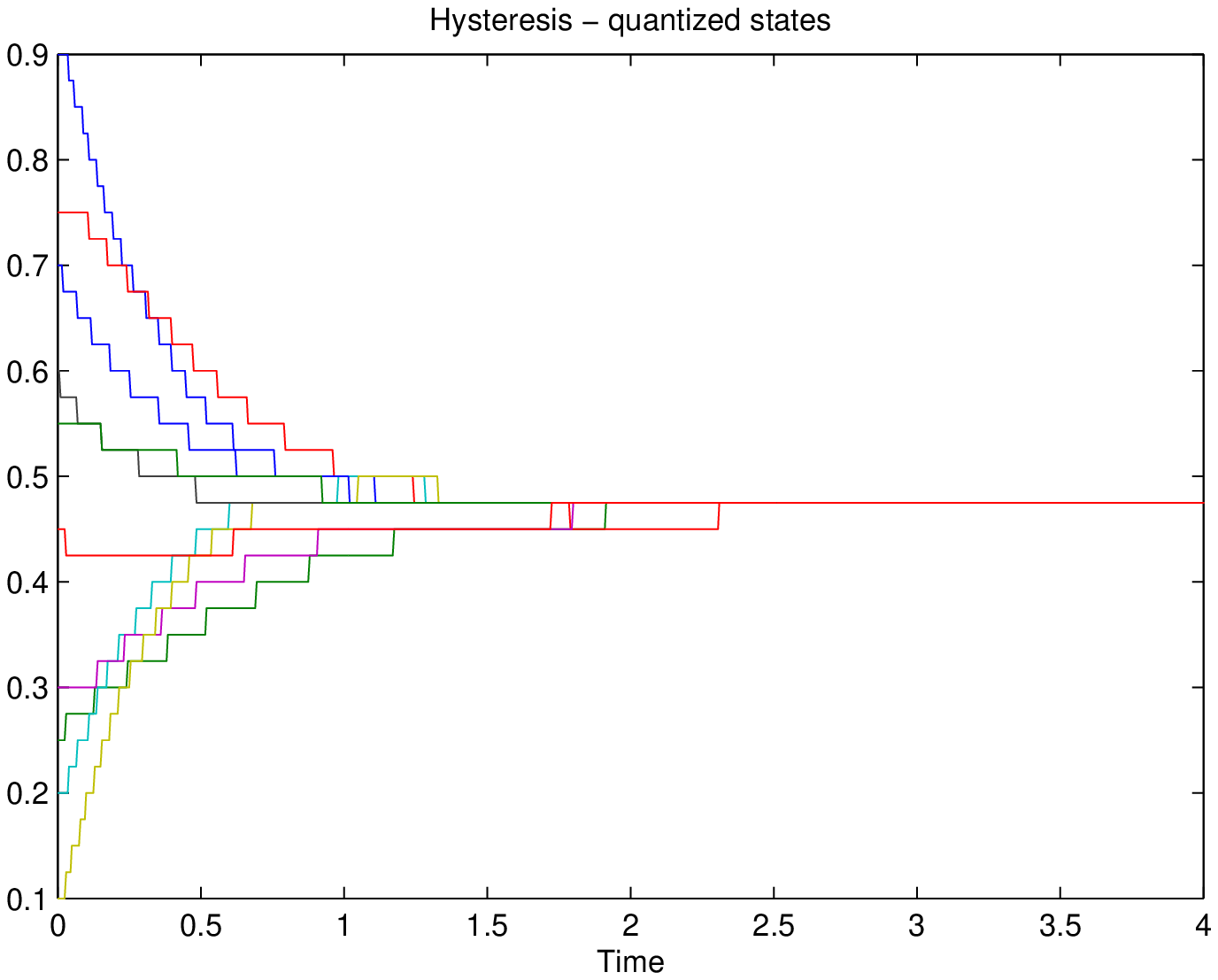}
\caption{Examples of evolutions of quantized dynamics on the graph
in Figure~\ref{fig:RGGex-cluster}.} \label{fig:RGGsolutions1}
\end{figure}
\begin{figure}[htb]
\includegraphics[width=.49\textwidth]{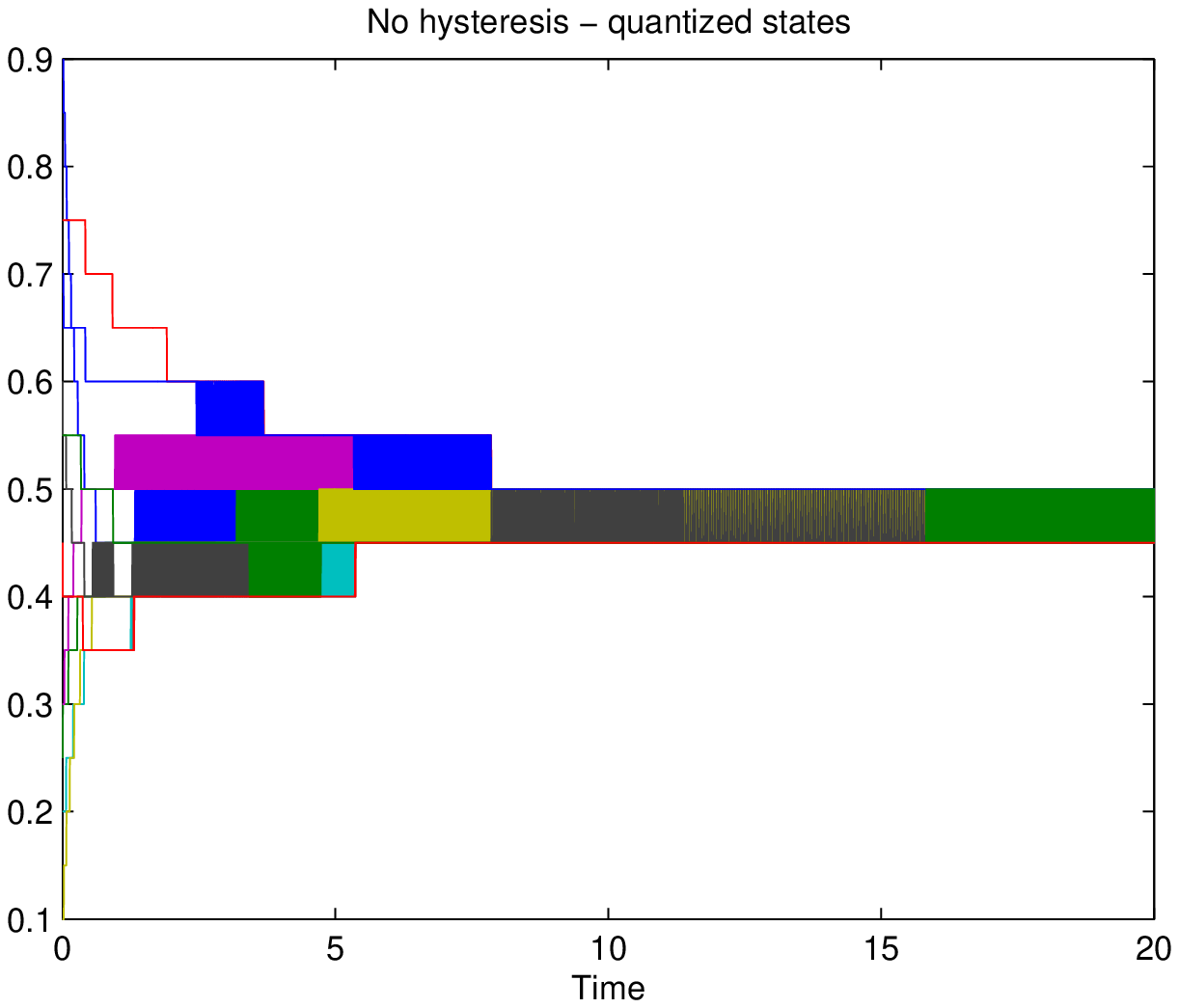}
\includegraphics[width=.49\textwidth]{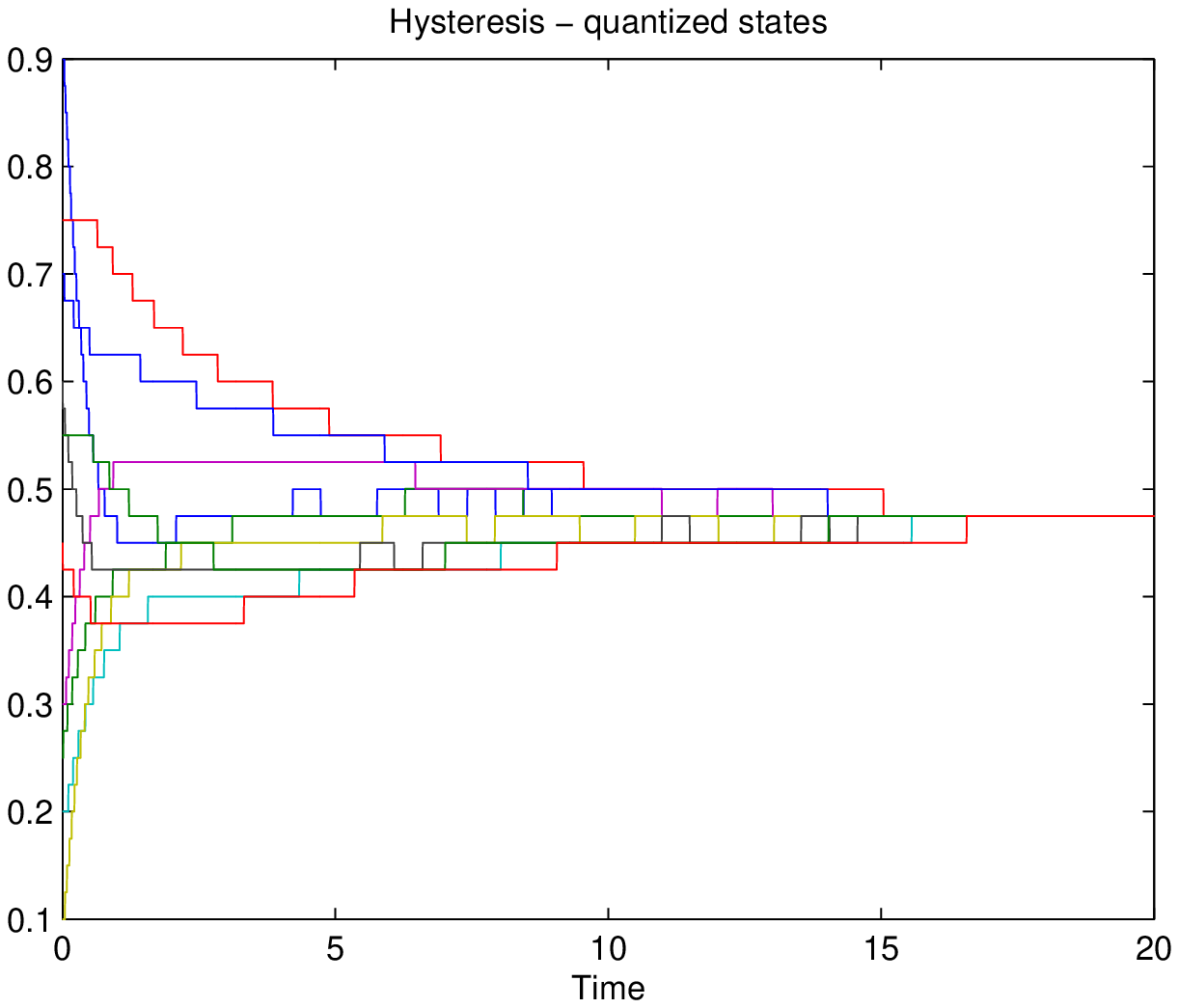}
\caption{Examples of evolutions of quantized dynamics on the graph
in Figure~\ref{fig:RGGex-tree}.} \label{fig:RGGsolutions2}
\end{figure}
To provide an example of limit cycles, we consider a {\em directed ring} topology, that is a graph such that $\neigh{i}=\{i+1\}$ for $i\in\until{N-1}$ and $\neigh{N}={1}.$%
 Figure~\ref{fig:RingSolutions}
shows that in this case hysteresis induces a limit cycle of
amplitude~$2\Delta$ on the quantized states. Furthermore, simulations demonstrate that larger
cycles can be obtained on larger rings.
\begin{figure}[htb]
\includegraphics[width=.49\textwidth]{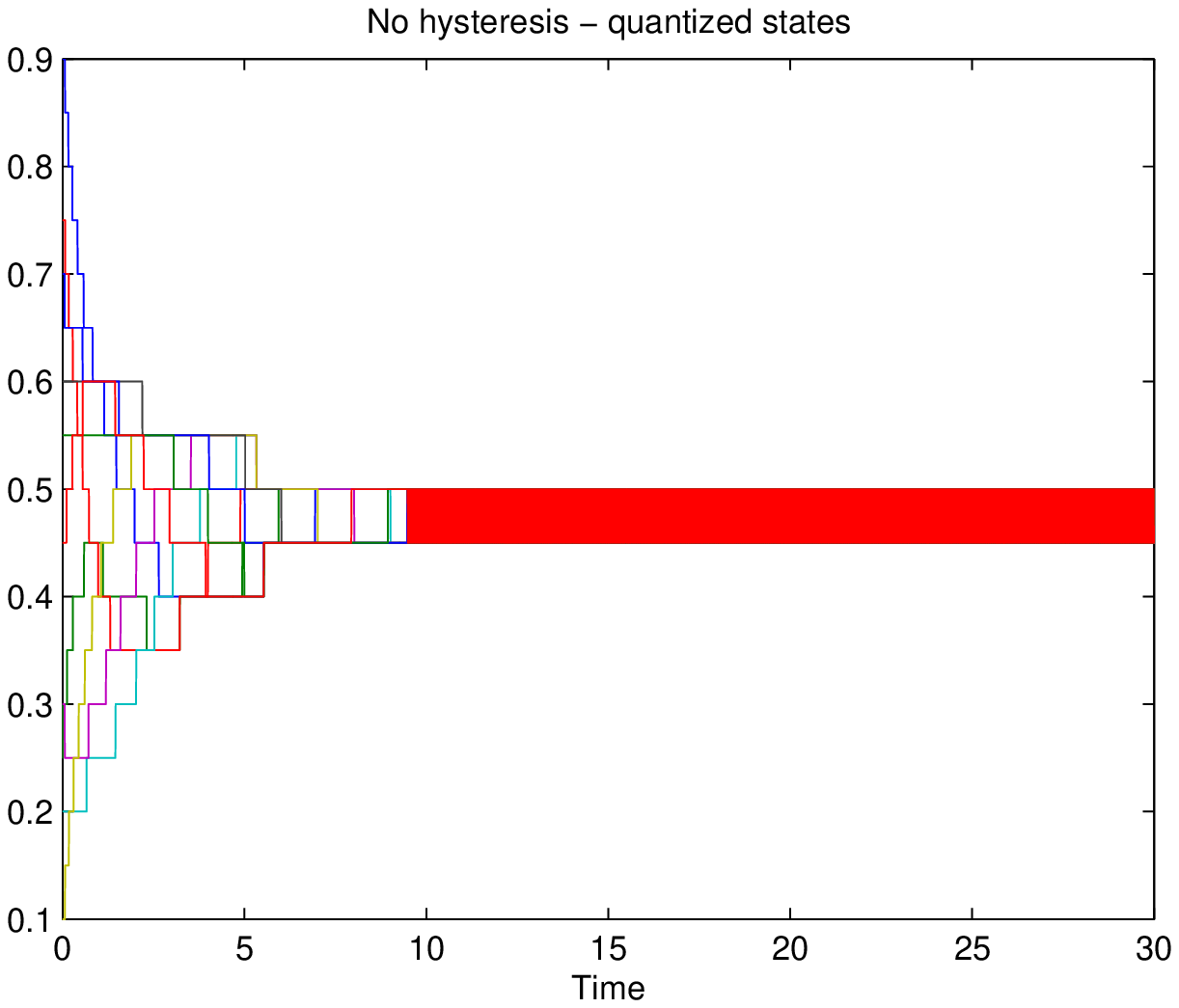}
\includegraphics[width=.49\textwidth]{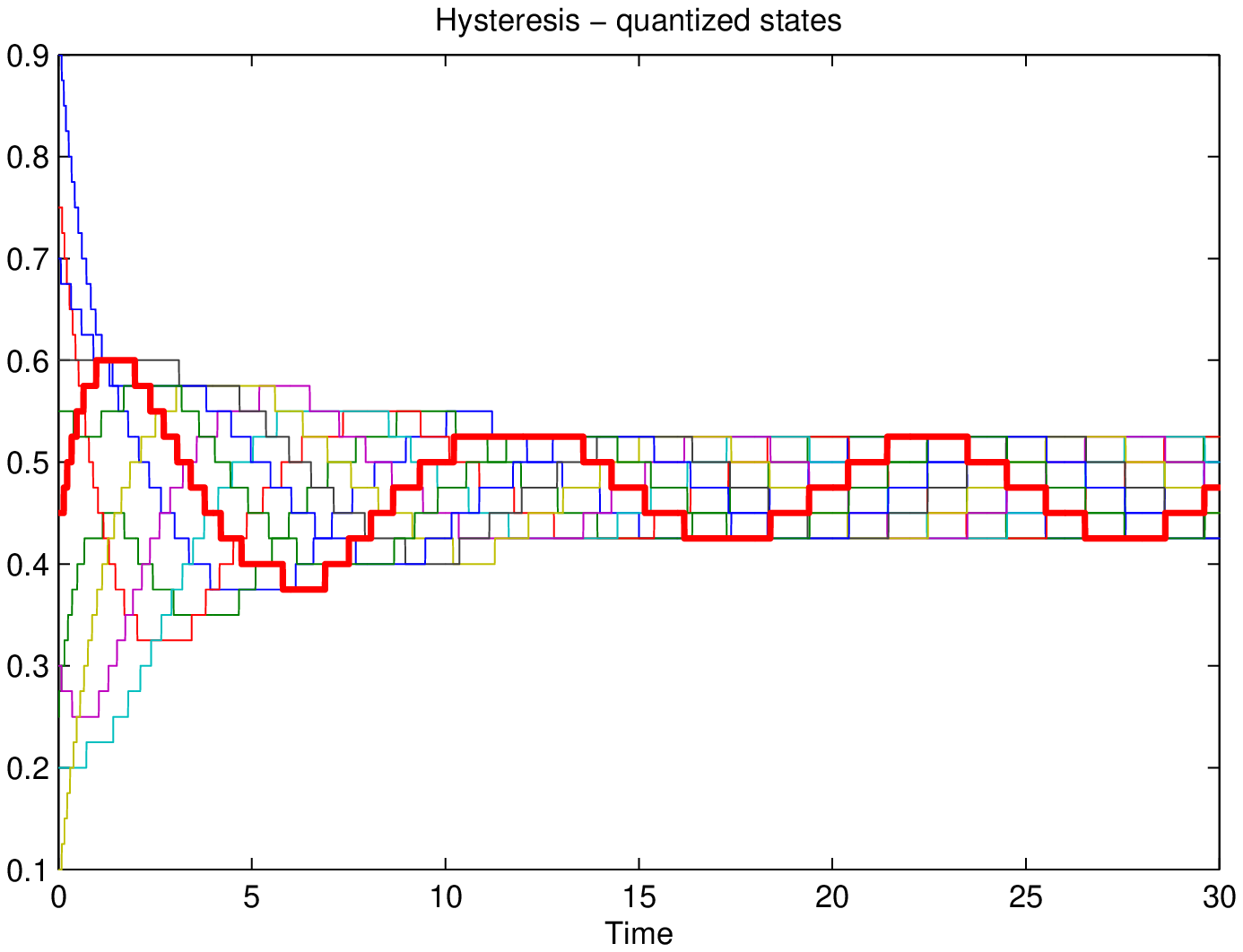}
\caption{Examples of evolutions of quantized dynamics on a
directed ring.} \label{fig:RingSolutions}
\end{figure}

\section{Conclusion}\label{sec:Outro}
In this paper we analyzed a consensus problem subject to quantized
communication, in terms of the stabilization of a dynamical system
by means of a discontinuous feedback. Two quantization rules have
been considered: a uniform static quantizer, and a hysteretic
quantizer, which we designed with the goal of avoiding chattering
phenomena. The convergence properties of both the resulting
systems have been studied, and illustrated through simulations.
Future work should include the analysis of other feedback consensus
dynamics with quantization effects, and the application of these techniques to the
problem of guaranteeing stable flocking of autonomous vehicles via
quantized feedback control.


\begin{thebibliography}{10}

\bibitem{JPA-AC:84}
J.~P. Aubin and A.~Cellina.
\newblock {\em Differential inclusions}, volume 264 of {\em Grundlehren der
  Mathematischen Wissenschaften}.
\newblock Springer, Berlin, 1984.

\bibitem{TCA-MJC-MGR:08}
T.~C. Aysal, M.~J. Coates, and M.~G. Rabbat.
\newblock Distributed average consensus with dithered quantization.
\newblock {\em IEEE Transactions on Signal Processing}, 56(10):4905--4918,
  2008.

\bibitem{AB-FC:99}
A.~Bacciotti and F.~Ceragioli.
\newblock Stability and stabilization of discontinuous systems and nonsmooth
  {L}iapunov functions.
\newblock {\em {ESAIM:} Control, Optimisation \& Calculus of Variations},
  4:361--376, 1999.

\bibitem{FB-JC-SM:09}
F.~Bullo, J.~Cort{\'e}s, and S.~Mart{\'\i}nez.
\newblock {\em Distributed Control of Robotic Networks}.
\newblock Applied Mathematics Series. Princeton University Press, 2009.

\bibitem{RC-FB-SZ:10}
R.~Carli, F.~Bullo, and S.~Zampieri.
\newblock Quantized average consensus via dynamic coding/decoding schemes.
\newblock {\em International Journal of Robust and Nonlinear Control},
  20(2):156--175, 2010.

\bibitem{RC-FF-PF-SZ:10}
R.~Carli, F.~Fagnani, P.~Frasca, and S.~Zampieri.
\newblock Gossip consensus algorithms via quantized communication.
\newblock {\em Automatica}, 46(1):70--80, 2010.

\bibitem{JC:06b}
J.~Cort{\'e}s.
\newblock Finite-time convergent gradient flows with applications to network
  consensus.
\newblock {\em Automatica}, 42(11):1993--2000, 2006.

\bibitem{JC:08-csm}
J.~Cort{\'e}s.
\newblock Discontinuous dynamical systems -- a tutorial on solutions, nonsmooth
  analysis, and stability.
\newblock {\em {IEEE} Control Systems Magazine}, 28(3):36--73, 2008.

\bibitem{JC:08}
J.~Cort{\'e}s.
\newblock Distributed algorithms for reaching consensus on general functions.
\newblock {\em Automatica}, 44(3):726--737, 2008.

\bibitem{KD:92}
K.~Deimling.
\newblock {\em Multivalued Differential Equations}.
\newblock De Gruyter, Berlin, 1992.

\bibitem{DVD-KHJ:10}
D.~V. Dimarogonas and K.~H. Johansson.
\newblock Stability analysis for multi-agent systems using the incidence
  matrix: Quantized communication and formation control.
\newblock {\em automatica}, 46(4):695--700, 2010.

\bibitem{AFF:88}
A.~F. Filippov.
\newblock {\em Differential Equations with Discontinuous Righthand Sides},
  volume~18 of {\em Mathematics and Its Applications}.
\newblock Kluwer Academic Publishers, 1988.

\bibitem{PF-RC-FF-SZ:08}
P.~Frasca, R.~Carli, F.~Fagnani, and S.~Zampieri.
\newblock Average consensus on networks with quantized communication.
\newblock {\em International Journal of Robust and Nonlinear Control},
  19(16):1787--1816, 2009.

\bibitem{RG-RS-AT:09}
R.~Goebel, R.~G. Sanfelice, and A.~R. Teel.
\newblock Hybrid dynamical systems.
\newblock {\em {IEEE} Control Systems Magazine}, 29(2):28--93, 2009.

\bibitem{OH:79}
O.~H{\'a}jek.
\newblock Discontinuous differential equations {I}.
\newblock {\em Journal of Differential Equations}, 32:149--170, 1979.

\bibitem{AK-TB-RS:07}
A.~Kashyap, T.~Ba{\c s}ar, and R.~Srikant.
\newblock Quantized consensus.
\newblock {\em Automatica}, 43(7):1192--1203, 2007.

\bibitem{TL-MF-LX-JFZ:09}
T.~Li, M.~Fu, L.~Xie, and J.-F. Zhang.
\newblock Distributed consensus with limited communication data rate.
\newblock {\em IEEE Transactions on Automatic Control}, 56(2):279--292, 2011.

\bibitem{AN-AO-AO-JNT:09}
A.~Nedic, A.~Olshevsky, A.~Ozdaglar, and J.~N. Tsitsiklis.
\newblock On distributed averaging algorithms and quantization effects.
\newblock {\em IEEE Transactions on Automatic Control}, 54(11):2506--2517,
  2009.

\bibitem{VNVS:60}
V.~V. Nemytskii and V.~V. Stepanov.
\newblock {\em Qualitative Theory of Differential Equations}.
\newblock Dover, New York, 1960.

\bibitem{BP-SSS:87}
B.~Paden and S.~S. Sastry.
\newblock A calculus for computing {F}ilippov's differential inclusion with
  application to the variable structure control of robot manipulators.
\newblock {\em IEEE Transactions on Circuits and Systems}, 34(1):73--82, 1987.

\bibitem{MP:03}
M.~Penrose.
\newblock {\em Random Geometric Graphs}.
\newblock Oxford Studies in Probability. Oxford University Press, 2003.

\bibitem{JY-SMLV-DL:08}
J.~Yu, S.~M. LaValle, and D.~Liberzon.
\newblock Rendezvous without coordinates.
\newblock In {\em {IEEE} Conf. on Decision and Control}, pages 1803--1808,
  Canc\'un, M\'exico, December 2008.

\end{thebibliography}

\end{document}